\documentclass[
]{amsart}
\usepackage{amssymb, amsthm, amsmath, pifont} 
\usepackage[usenames]{color}
\usepackage{amsfonts}
\usepackage{enumerate}
\usepackage{pdfpages}
\usepackage[all]{xypic}
\usepackage{graphicx}
\usepackage{psfrag}
\usepackage{verbatim}
\usepackage{pstricks}
\usepackage[latin1]{inputenc}
\usepackage{mathrsfs}
\usepackage[all,knot]{xy}
\newtheorem{theorem}{Theorem}[section]

\newtheorem{corollary}[theorem]{Corollary}

\newtheorem{lemma}[theorem]{Lemma}
\newtheorem{question}[theorem]{Question}
\newtheorem{example}[theorem]{Example}
\newtheorem{definition}[theorem]{Definition}
\newtheorem*{intdef}{Definition}
\newtheorem{remark}[theorem]{Remark}
\newtheorem{prop}[theorem]{Proposition}

\newtheorem{claim}[theorem]{Claim}
\newtheorem{obs}[theorem]{Observation}

\newenvironment{Example}{\begin{example}\rm}{\end{example}}
\newenvironment{Definition}{\begin{definition}\rm}{\end{definition}}

\newenvironment{Remark}{\begin{remark}\rm}{\end{remark}}
\newenvironment{Claim}{\begin{claim}\rm}{\end{claim}}
\newenvironment{Obs}{\begin{obs}\rm}{\end{obs}}
\newenvironment{Question}{\begin{question}\rm}{\end{question}}
\def\et{\;\mbox{and}\;}

\def\s{{\sigma}}
\def\so{{\sigma_\omega}}

\def\So{\widetilde{\sigma_\omega}}

\def\u{\upsilon}
\def\U{\Upsilon}
\def\Ut{\Upsilon_t}
\def\HUt{
\widetilde{\Upsilon(t)}}

\def\l{{{\ell}}}

\def\for{\quad\mbox{for }}
\def\all{\quad\mbox{all }}

\def\et{\quad\mbox{and}\quad}

\def\epsilon{\varepsilon}
\def\N{\mathbb{N}}

\def\R{\mathbb{R}}
\def\Z{\mathbb{Z}}
\def\C{\mathbb{C}}

\begin{document}
\title
{On cobordisms between knots, braid index, and the Upsilon-invariant}
\author{Peter Feller}

\address{Max Planck Institute for Mathematics, Vivatsgasse 7, 53111 Bonn, Germany}
\email{peter.feller@math.ch}
\author{David Krcatovich}

\address{Rice University, Department of Mathematics, Houston, TX 77047, US}
\email{dk27@rice.edu}

\thanks{The first author gratefully acknowledges support by the Swiss National Science Foundation Grant 155477. The second author was partially supported by NSF grant DMS-1309081.}
\begin{abstract}
We use 
Ozsv\'ath, Stipsicz, and Szab\'o's Upsilon-invariant to provide bounds on cobordisms between knots that `contain full-twists'. 
In particular, we recover and generalize a classical consequence of the Morton-Franks-Williams inequality for knots: positive braids that contain a positive full-twist realize the braid index of their closure. We also establish that quasi-positive braids that are sufficiently twisted realize the minimal braid index among all knots that are concordant to their closure. Finally, we 
provide inductive formulas for the Upsilon-invariant of torus knots and compare it to the Levine-Tristram signature profile.
\end{abstract}
\maketitle
\section{Introduction}
This article is concerned with the study of \emph{knots} in the 3-sphere $S^3$---smooth oriented embeddings of the circle $S^1$ considered up to ambient isotopy. A classical theorem of Alexander states that every knot arises as the closure of an $n$-braid for some positive integer $n$~\cite{Alexander_23_ALemmaOnSystemsOfKnottedCurves}. Here an \emph{$n$-braid} is an element of Artin's braid group on $n$-strands $B_n$~\cite{Artin_TheorieDerZoepfe}. Alexander's result naturally leads to the definition of the \emph{braid index} of a knot $K$---the minimal positive integer $n$ such that $K$ arises as the closure of an $n$-braid. In general, the braid index is difficult to compute even for simple families of knots.

The main result of this article relates the braid indices of two knots to the minimal genus of cobordisms between them---their \emph{cobordism distance}. As a consequence, we reprove and generalize a classical consequence of the Morton-Franks-Williams inequality~\cite{Morton_86,Franks_Williams_87_BraidsAndTheJonesPolynomial} and we find that most torus knots minimize the braid index among all knots concordant to them; see Theorem~\ref{thm:concbraidindex}.
In fact, to the authors' knowledge these constitute the first example of an infinite family of concordance classes for which the minimal braid index is unbounded. Our results also yield obstructions to optimal and algebraic cobordisms between knots, which was the original motivation for our study.
Before stating our results, we recall some notions surrounding knot concordance.

For a knot $K$ in $S^3$, let $g_4(K)$ denote the \emph{slice genus}\textemdash the minimal genus of a properly embedded smooth oriented
surface in $B^4$ with boundary $K$. This generalizes the notion of a \emph{slice knot}---a knot $K$ with $g_4(K)=0$---due to Fox. More generally, the cobordism distance $d(K,T)$ between two knots $K,T$ is defined by $g_4(K\sharp m(T))$, where $\sharp$ denotes the connected sum of knots and $m(T)$ is the mirror image of $T$ with reversed orientation. Knots $K$ and $T$ are said to be \emph{concordant} if $g_4(K\sharp m(T))=0$. While this notion depends on the orientation of the involved knots, the invariants dealt with in this paper are preserved under orientation reversal, so we will hereafter neglect to mention orientations.
The name `cobordism distance' is justified by the fact that $d$ descends to a metric on the concordance group
\[\mathfrak{C}=\{\text{isotopy classes of oriented knots}\}/K\sim T \text{ iff }g_4(K\sharp m({T}))=0.\]
In particular, for all knots $K$ and $T$, $d$ satisfies the 
triangle inequality
\begin{equation}\label{eq:tri}|g_4(K)-g_4(T)|\leq d(K,T)=g_4(K\sharp m({T}))\leq g_4(K)+g_4(T)
.\end{equation}

In this text, we use Ozsv\'ath, Stipsicz, and Szab\'o's
$\Upsilon$-invariant to improve the triangle inequality~\eqref{eq:tri} by a term depending on the braid indices when the involved knots are \emph{quasi-positive}
---knots $K$ for which there exist positive integers $n$ and $l$ such that $K$ is the braid closure $\widehat{\beta}$ of an $n$-braid $\beta$
given as the product of $l$ conjugates of the standard generators $a_i$ of the braid group on $n$ strands.
Quasi-positive knots are a natural class of knots (they are precisely the knots that arise as transversal intersections of algebraic curves in $\C^2$ with the unit sphere $S^3$~\cite{Rudolph_83_AlgFunctionsAndClosedBraids,BoileauOrevkov_QuasiPositivite}) that generalize positive knots, braid positive knots, algebraic knots, and torus knots. 
Based on the local Thom conjecture~\cite{KronheimerMrowka_Gaugetheoryforemb}, Rudolph establishes that
the slice genus of a quasi-positive knot $K$ is given as
$g_4(K)=\frac{l-n+1}{2}$~\cite{rudolph_QPasObstruction}.
Knowing $g_4(K)$ and $g_4(L)$ makes bounds for $d(K,L)$ in terms of $g_4(K)$ and $g_4(L)$ interesting.


\begin{theorem}\label{thm:unknottingtotorusknotandbraidindex}
Let $K$ and $L$ be quasi-positive knots, denote the braid index of $K$ by $m$.
If $L$ is the closure of an $n$-braid of the form $(a_1a_2\cdots a_{n-1})^{nk+1}\alpha$, where $n$ and $k$ are positive integers and $\alpha$ is a quasi-positive braid,
then
\begin{equation}\label{eq:imprtri}k(n-m)+g_4(K)-g_4(L)\leq d(K,L)
.\end{equation}
The inequality~\eqref{eq:imprtri} also holds if $L$ is the closure of an $n$-braid $(\Delta^2)^k\alpha$, where $\alpha$ is a positive braid and $\Delta^2$ denotes the full-twist $(a_1a_2\cdots a_{n-1})^{n}$.
\end{theorem}
Theorem~\ref{thm:unknottingtotorusknotandbraidindex} unifies and generalizes two types of consequences.
Firstly, Theorem~\ref{thm:unknottingtotorusknotandbraidindex} obstructs the existence of \emph{optimal cobordisms}\textemdash cobordisms of genus equal to the difference of the slice genera of the involved knots\textemdash between quasi-positive knots $K$ and $L$ when $K$ has strictly smaller braid index than $L$, $g_4(K)\geq g_4(L)$, and $L$ is as in Theorem~\ref{thm:unknottingtotorusknotandbraidindex}. \emph{Algebraic cobordisms}\textemdash cobordisms given by the intersection of smooth algebraic curves in $\mathbb{C}^2$ with $\{(x,t)\in\C^2\ |\ a^2\leq|x|^2+|y|^2\leq b^2\}\cong S^3 \times [a,b]$\textemdash are optimal cobordisms by the Thom conjecture~\cite{KronheimerMrowka_Gaugetheoryforemb}, and are therefore also obstructed by Theorem~\ref{thm:unknottingtotorusknotandbraidindex}.
A special case that is of interest is when $K$ and $L$ are algebraic knots. In this case, Wang~\cite{Wang} independently established that no optimal cobordisms exist. 
Secondly, Theorem~\ref{thm:unknottingtotorusknotandbraidindex} can detect the braid index of knots. Indeed, by considering concordant $K$ and $L$, i.e.~$d(K,L)=g_4(K)-g_4(L)=0$, Theorem~\ref{thm:unknottingtotorusknotandbraidindex} yields
\begin{corollary}\label{cor:concordancebraidindexdetection}For all knots $L$ as in Theorem~\ref{thm:unknottingtotorusknotandbraidindex}, all quasi-positive knots $K$ concordant to $L$ have braid index at least $n$.\qed\end{corollary}
We also show 
that sufficiently twisted quasi-positive $n$-braid closures cannot be concordant to \textit{any} knot of smaller braid index:
\begin{theorem}\label{thm:concbraidindex}
Let $L$ be the closure of an $n$-braid $\beta,$ where $\beta = \Delta^{2k}a_1\cdots a_{n-1} \alpha$ for some quasi-positive braid $\alpha$. If $k\geq n-1$, then any knot concordant to $L$ has braid index at least $n$. 
\end{theorem}
This says in particular that
, for coprime positive integers $i$ and $p$, the $(p,p(p-1)+i)$-torus knot $T_{p,p(p-1)+i}$ is not concordant to any $(p-1)$-braid closure.

Corollary~\ref{cor:concordancebraidindexdetection} generalizes and reproves (by setting $K$ to be $L$) the following result on the braid index of positive braids by Franks and Williams.\footnote{Strictly speaking this is not a generalization since our setting is restricted to knots, while the original result holds for all links.}
\begin{corollary}\cite[(2.4)~Corollary]{Franks_Williams_87_BraidsAndTheJonesPolynomial}\label{cor:FranksWilliams}
If a knot $L$ is the closure of an $n$-braid $\beta=\Delta^2\alpha$, where $\alpha$ is a positive braid, then the braid index of $L$ is $n$.\qed
\end{corollary}
The original proof of Corollary~\ref{cor:FranksWilliams} makes use of the Morton-Franks-Williams inequality~\cite{Morton_86,Franks_Williams_87_BraidsAndTheJonesPolynomial}, which relates the breadth of the HOMFLY polynomial with the braid index.
The present proof is based on the concordance invariant $\Upsilon$ and allows to extend Corollary~\ref{cor:FranksWilliams} (at least partially) to results about the concordance class such as Corollary~\ref{cor:concordancebraidindexdetection} and Theorem~\ref{thm:concbraidindex}, which, to the authors' knowledge, are the first results of this type. It is natural to ask whether the quasi-positivity assumption on $K$ in Corollary~\ref{cor:concordancebraidindexdetection} can be dropped. 
In Section~\ref{sec:questions}, we ask whether a `concordance generalized Jones conjecture' holds; compare Question~\ref{qu:concordanceGenJonesConj}. If yes, this would imply that Corollary~\ref{cor:concordancebraidindexdetection} holds true without any assumption on $K$, and, in particular, {that each torus knot realizes the minimal braid index among all knots in its concordance class}.

In a different direction, we ask whether $\Upsilon$ provides a lower bound for the braid index as follows. If $b$ denotes the braid index of $K$, is $\Upsilon_K$ linear on $[0,\frac{2}{b}]$; compare Question~\ref{qu:bindex}. A positive answer to this question would also imply that the quasi-positivity assumption on $K$ in Corollary~\ref{cor:concordancebraidindexdetection} can be dropped.
In fact, a positive answer would imply that in the concordance group, the torus knot $T_{p,p+1}$ is linearly independent from the subgroup generated by all knots which are closures of braids of braid index less than $p$. This would allow one to filter the concordance group by `concordance braid index', and perhaps better understand its structure.

We prove Theorem~\ref{thm:unknottingtotorusknotandbraidindex} by combining three ingredients. Firstly, the calculation of $\Upsilon$ for the torus knots of the form $T_{n,kn+1}$. In fact, we provide an inductive formula for $\Upsilon$ for all torus knots, which might be of independent interest; see Section~\ref{sec:Upsfortorusknots}. Secondly, we observe that $\Upsilon$ can be used to prove the slice-Bennequin inequality, \`a la Rudolph \cite{rudolph_QPasObstruction}. The connection to the braid index arises as follows: for positive integers $n$, the quantity $\left|\frac{\Ut}{t}\right|$ for $t\leq\frac{2}{n}$ can be used to prove the slice-Bennequin inequality, while for $t\in (\frac{2}{n},2-\frac{2}{n})$ this is not the case. As a third ingredient we use the generalized Jones conjecture as proven 
by Dynnikov and Prasolov~\cite{DynnikovPrasolov_13} and, independently, by LaFountain and Menasco~\cite{LaFountainMenasco_14}.
The ingredients are combined as follows. We introduce notions that measure how many full-twists a knot `contains', which turns out to yield a good framework to prove Theorem~\ref{thm:triimprovmentforknotscontainingtwists}---a generalization of Theorem~\ref{thm:unknottingtotorusknotandbraidindex}. For this the calculation of $\Upsilon$ for $T_{n,kn+1}$ is used and parts of the proof (i.e.~the proof of Proposition~\ref{prop:qpbraidindexdetection}) mimic a proof of the slice-Bennequin inequality.
We invoke the generalized Jones conjecture to show that quasi-positive knots fit well into this setting and that Theorem~\ref{thm:triimprovmentforknotscontainingtwists} implies Theorem~\ref{thm:unknottingtotorusknotandbraidindex}. All of this, as well as the proof of Theorem~\ref{thm:concbraidindex}, is done in Section~\ref{sec:upstorusknotsandmainresult}.

In Section~\ref{sec:homogenization}, we study $\Upsilon$ from a coarser point of view and compare it with the Levine-Tristram signature profile. Before we make this precise, we mention the following examples that came out of this study.
For a non-negative integer $n$, let $K_{n}$ be the closure of the positive $3$-braid $(a_1^2a_2^2)^n(a_1a_2)$.
For $n\geq 6$, $\Upsilon_{K_{n}}$ is not convex (compare Example~\ref{Ex:a1a1a2a2}), which shows that $K_{n}$ is not an $L$-space knot (since $L$-space knots have convex $\Upsilon$~\cite{BorodzikHedden}), while, for $n\equiv 1 \mod 3$, these $K_n$ pass all known classical criteria for $L$-space knots: they are fibred and strongly quasi-positive (since they are closures of positive 3-braids), and their Alexander polynomials satisfy the criteria established in~\cite{OzsvathSzabo_03_AbsolutlyGradedFloerHomologies,HeddenWatson,Krcatovich}. Also, these examples provided a negative answer to the question whether $\Upsilon$ of positive braid closures is always convex by Borodzik and Hedden~\cite{BorodzikHedden}.

To make our coarse point of view precise, we use homogenized invariants.
We fix a positive integer $n$. For a real valued link invariant $\tau$ and any $n$-braid $\beta$, we set
\[\widetilde{\tau}(\beta)=\lim_{l\to\infty}\frac{\tau(\widehat{\beta^l})}{l}.\]
A slight variation also yields a notion of homogenization when $\tau$ is a knot invariant, and it turns out that $\widetilde{\tau}$ is well-defined for both $\U(t)$ and $\so$, where the latter denotes the Levine-Tristram signature corresponding to a unit complex number $\omega$.
We calculate that for all $n$-braids $\beta$ and $t\leq\frac{2}{n}$ the homogenization $2\widetilde{\U(t)}(\beta)$ equals the homogenization $\widetilde{\s_{e^{t\pi i}}(\beta)}$ and that
for all $3$-braids $\beta$ the homogenization of the signature $\widetilde{\s}(\beta)$ equals $2\widetilde{\Upsilon(1)}(\beta)$. The latter yields
$|\Upsilon_K(1)-\frac{\s(K)}{2}|\leq 2$ for all knots $K$ which are closures of $3$-braids. This is of interest because $|\Upsilon_K(1)-\frac{\s(K)}{2}|$ is a lower bound for the smooth four-dimensional
crosscap number~\cite{OSS_2015}.
On the other hand, we provide a family of $3$-braids on which $|\HUt-\frac{\widetilde{\s_{e^{t\pi i}}}}{2}|$ is arbitrarily large, for $t=\frac{3}{4}$.

{\bf{Acknowledgments}:} We thank Sebastian Baader, 
Matt Hedden, Lukas Lewark, and Aru Ray for helpful discussions. Thanks also to Peter Ozsv\'ath for pointing us to Dan Dore's~\cite{Dore15}.
We owe special thanks to Maciej Borodzik, who referred us to~\cite[Proposition 5.2.4]{BodnarNemethi},
which we need to compute $\Upsilon$ of torus knots in full generality.
Finally, many thanks to the anonymous referee for their detailed and on point suggestions.

\section{$\Upsilon$ for torus knots}\label{sec:Upsfortorusknots}

In~\cite{OSS_2014}, the smooth concordance invariant $\Upsilon_K$ is defined from a `$t$-modified knot Floer homology'. The invariant takes the form of a continuous piecewise linear function $[0,2]\to \R$ whose derivative has finitely many discontinuities. It is additive under connected sum, and for each value of $t>0$, $\Upsilon$ bounds the slice genus:
\begin{equation}\label{eqn:upsgenusbound}
\Big| \frac{\Upsilon_K(t)}{t} \Big| \leq g_4(K).
\end{equation}

In~\cite[Theorem 1.15]{OSS_2014}, it is shown how for a torus knot (or any $L$-space knot, more generally), $\Upsilon$ can be obtained from the Alexander polynomial. This is carried out explicitly for the case of the $(n,n+1)$-torus knots.
\begin{prop}[{\cite[Proposition 6.3]{OSS_2014}}\footnote{While the application of \cite[Theorem 1.15]{OSS_2014} yields the result stated here, the preprint available as of this writing contains a small typo in Proposition 6.3 (the index $i$ is shifted by $1$), which accounts for the discrepancy between the result quoted here and that written in~\cite{OSS_2014}.}]\label{prop:OSStorus}
Consider the torus knot $T_{n,n+1}$.
For any $t\in \left[ \frac{2i}{n},\frac{2i+2}{n}\right]$,
\[ \;\quad\quad\quad\quad\quad\quad\quad\quad\Upsilon_{T_{n,n+1}} (t)= -i(i+1) -\frac{1}{2}n(n-1-2i)t.\quad\quad\quad\quad\quad\quad\quad\qed\]

\end{prop}

In this section we give an explicit formula for $\Upsilon$ of all positive torus knots, showing it is always a sum of the functions $\Upsilon_{T_{n,n+1}}$, for varying $n$. Indeed, the same can be seen to be true for any algebraic knot. Though we will only need Corollary~\ref{cor:upsfortnnk+1} to obstruct the existence of optimal cobordisms, we prove the following more general statement, which was independently conjectured by Dore~\cite[Conjecture~1]{Dore15}.

\begin{prop}\label{prop:upsalltorus}
Let $a<b$ be two coprime positive integers. Let $q_i$ and $r_i$ be the $i$th quotient and remainder occurring in the Euclidean algorithm for $b$ and $a$ (so that $r_0=a$ and $r_{i-1}=q_{i}r_{i}+r_{i+1}$). Then \[ \Upsilon_{T_{a,b}}(t) = \sum_{i\geq 0} q_i \cdot \Upsilon_{T_{r_i,r_i+1}}(t).\]
\end{prop}

In other words, $\Upsilon_{T_{a,b}}(t)$ can be calculated inductively by using
\[\Upsilon_{T_{a,b}}(t)=\Upsilon_{T_{a,b-a}}(t)+\Upsilon_{T_{a,a+1}}(t)\]
and Proposition~\ref{prop:OSStorus}. This looks very similar to the inductive scheme for the calculation of the signature provided by Gordon, Litherland, and Murasugi~\cite{GLM}. In fact, it turns out that the signature (and its generalization the Levine-Tristram signatures) and $\Upsilon$ are surprisingly close for torus knots; see Corollary~\ref{cor:u=sfortoruslinks}.

\begin{corollary}\label{cor:upsfortnnk+1} For positive integers $n$ and $k$, we have
\begin{align*}
\quad\;\U_{T_{n,nk+1}}(t)=& k\cdot \U_{T_{n,n+1}}(t) &\\
=& k\left(-(i+1)i-\frac{1}{2}n(n-1-2i)t\right)\for t\in\left[\frac{2i}{n},\frac{2i+2}{n}\right].\quad&\hfill\qed
\end{align*}
\end{corollary}
Corollary~\ref{cor:upsfortnnk+1} can also be obtained by directly working with the combinatorial description given in~\cite[Theorem 1.15]{OSS_2014} as done by Dore~\cite[Theorem~4]{Dore15}.
In particular, Corollary~\ref{cor:upsfortnnk+1} yields
\begin{Obs}\label{Obs:UtforTpnp1} For a torus knot $T=T_{n,nk+1}$ with $n$ and $k$ positive integers,  one has
\[\U_T(t)=-t\tau(T)=-tg_4(T) \for t\leq\frac{2}{n} \et \]
\[\U_T(t)\geq-tg_4(T)+k(n t-2)>-tg_4(T) \for \frac{2}{n}<t\leq 1.\]

\end{Obs} Here, $\tau$ denotes the Ozsv\'ath--Szab\'o concordance invariant, whose value for positive torus knots equals the slice genus \cite{OzsvathSzabo_03_KFHandthefourballgenus}.

The rest of this section is concerned with the proof of Proposition~\ref{prop:upsalltorus}. Only Corollary~\ref{cor:upsfortnnk+1} and Observation~\ref{Obs:UtforTpnp1} are used in the rest of the text.

The proof of Proposition~\ref{prop:upsalltorus} relies on two key ideas: that $\Upsilon$ for an algebraic knot can be computed from a semigroup counting function, and that this counting function behaves well under blowups of singularities.

An algebraic knot $K$ can be realized as the link of a singularity of an algebraic curve in $\mathbb{C}^2$. Associated to the singularity is a semigroup of non-negative integers (see \cite{Wall} for a detailed exposition) which we will denote $S_K$. Define the counting function of $S_K$ as \[ H_K(i) = \# \{ s \in S_K | s<i \}.\]

Computing $\Upsilon$ from the Alexander polynomial can be rephrased as computing it from the semigroup counting function.

\begin{prop}[{ \cite[Proposition 3.4]{BorodzikLivingston_13Ar}, cf.  \cite[Theorem 1.15]{OSS_2014}}]\label{prop:semigroups}
Let $K$ be an algebraic knot whose corresponding semigroup has counting function $H_K$, and let $g$ be the genus of $K$. For any $t\in [0,1]$, \[\;\;\Upsilon_K(t) = -2 \min_{i\in \{0,1,\ldots,2g \} } \left\{  H_K(i) + \frac{t}{2}(g-i)\right\}= -2 \min_{i\in \Z } \left\{  H_K(i) + \frac{t}{2}(g-i)\right\}. \qed\]
\end{prop}

Note that $\Upsilon_K(2-t)=\Upsilon_K(t)$ for all $t\in [0,2]$ \cite[Proposition 1.2]{OSS_2014}, so $\Upsilon_K$ is determined by its values on $[0,1]$. Next we note the effect that blowing up a singularity has on the corresponding semigroup. We thank Maciej Borodzik for pointing us to this result.

\begin{prop}[{\cite[Proposition 5.2.4]{BodnarNemethi}}]\label{prop:blowup}
Suppose $K_2$ is the link of a plane curve singularity with multiplicity $m$, and $K_1$ is the link of the singularity blownup once. Then \[\quad\quad\quad\quad\quad\quad\quad\quad\quad H_{K_2}(i) = \min_{j\in \Z} \left\{ H_{K_1}(i-j) + H_{T_{m,m+1}} (j)\right\}.\quad\quad\quad\quad\quad\quad\qed\]
\end{prop}

Finally, we note that this relation of semigroups corresponds to additivity of the $\Upsilon$ invariant. The following is implicit in the work of Borodzik and Livingston \cite{BorodzikLivingston_13Ar}. They stated their results in terms of `$J$-functions' of connected sums of algebraic knots, but Proposition \ref{prop:semigroups} gives a corresponding statement in terms of $\Upsilon$.

\begin{lemma}[Borodzik--Livingston]\label{lem:upssum}
If $K_1,K_2$ and $K_3$ are algebraic knots whose corresponding semigroups have counting functions related as \begin{equation} \label{eq:semigroupadd}
H_{K_3} (i) = \min_{j\in \Z} \{ H_{K_1}(i-j) +H_{K_2}(j) \},
\end{equation}
then \[ \Upsilon_{K_3}(t) = \Upsilon_{K_1}(t) + \Upsilon_{K_2}(t).\]
\end{lemma}
\begin{proof}

First note that if the counting functions satisfy \eqref{eq:semigroupadd}, then
\begin{equation}\label{eq:genusadd}
g(K_3)=g(K_2)+g(K_1).
\end{equation}
This is due to two facts. First, for large $N$ (specifically, $N\geq 2g(K)$), we have $H_K(N)=N-g(K).$ Second, it is clear from the counting function definition that $H_K(N-i) \geq H_K(N)-i$ for any $i$. Choosing $N$ to be larger than twice the genus of any of the three knots, we see
\begin{align*} H_{K_3}(2N)=& \min_{j\in \Z} \{ H_{K_1}(2N-j)+H_{K_2}(j)\} \\
=& H_{K_1}(N)+H_{K_2}(N) \\
=& N-g(K_1)+N-g(K_2).
\end{align*} On the other hand, \( H_{K_3}(2N)=2N-g_3(K),\) so \eqref{eq:genusadd} follows.\\
Now we have
\begin{align*}
\Upsilon_{K_3}(t) =& -2 \min_{i\in \Z} \left\{H_{K_3}(i)+\frac{t}{2}(g(K_3)-i) \right\}\\
=& -2\min_{i\in \Z} \left\{ \min_{j\in \Z} \left\{ H_{K_1}(i-j)+H_{K_2}(j) \right\} +\frac{t}{2}(g(K_1)+g(K_2)-i) \right\} \\
=& -2\min_{i,j\in \Z} \left\{  H_{K_1}(i-j)+\frac{t}{2}(g(K_1)+j-i)+H_{K_2}(j)+\frac{t}{2}(g(K_2)-j) \right\}\\
=& -2\min_{\ell,j\in \Z} \left\{  H_{K_1}(\ell)+\frac{t}{2}(g(K_1)-\ell)+H_{K_2}(j)+\frac{t}{2}(g(K_2)-j) \right\}\\
=& -2\min_{\ell\in \Z} \left\{  H_{K_1}(\ell)+\frac{t}{2}(g(K_1)-\ell)\right\} -2\min_{j\in \Z}\left\{H_{K_2}(j)+\frac{t}{2}(g(K_2)-j) \right\}\\
=& \Upsilon_{K_1}(t) +\Upsilon_{K_2}(t)
\end{align*}

\end{proof}

\begin{proof}[Proof of Proposition \ref{prop:upsalltorus}]

If $a=1$, then $T_{a,b}$ is the unknot, which has $\Upsilon\equiv 0$, and the statement is clearly true. Now assume that it holds for all torus knots $T_{n,b}$ for all $n\leq a-1$, and consider $T_{a,b}$, with $a<b$ coprime.

The torus knot $T_{a,b}$ is the link of the singularity of the curve \[ x^a-y^b=0.\] Since $a<b$, this singularity has multiplicity $a$. Write $b=qa+r$ with $0<r<a$, and the singularity can be blownup $q$ times to obtain the singularity \[ x^a - y^{b-qa}=0,\] whose link is $T_{a,b-qa}$. By applying Proposition \ref{prop:blowup} to the first blowup, we have  \[ H_{T_{a,b}}(i) = \min_{j\in \Z} \left\{ H_{T_{a,b-a}}(i-j) +   H_{T_{a,a+1}}(j) \right\} ,\] and therefore by Lemma \ref{lem:upssum}, \[ \Upsilon_{T_{a,b}}(t) = \Upsilon_{T_{a,b-a}}(t)+\Upsilon_{T_{a,a+1}}(t).\]
Repeating this for each of the $q$ blowups, we get \[ \Upsilon_{T_{a,b}}(t) = \Upsilon_{T_{a,b-qa}}(t)+q\cdot \Upsilon_{T_{a,a+1}}(t).\]
Now $T_{a,b-qa} = T_{a,r}$ is isotopic to $T_{r,a}$, and the result follows by the inductive assumption.

\end{proof}

\begin{example}\label{ex:torusupscomp}
Consider the torus knot $T_{8,11}$. The Euclidean algorithm gives
\begin{align*}
11 =& 1\cdot8 +3, \\
8=& 2\cdot3 +2,\\
3=& 1\cdot2 +1,
\end{align*}
and reasoning as in the proof of Proposition \ref{prop:upsalltorus} gives
\begin{align*}
\Upsilon_{T_{8,11}}(t) =& \Upsilon_{T_{8,9}}(t)+\Upsilon_{T_{8,3}}(t) \\
=& \Upsilon_{T_{8,9}}(t) + \Upsilon_{T_{3,8}}(t)\\
=&  \Upsilon_{T_{8,9}}(t) + \Upsilon_{T_{3,4}}(t) + \Upsilon_{T_{3,5}}(t)\\
=& \Upsilon_{T_{8,9}}(t) + \Upsilon_{T_{3,4}}(t) + \Upsilon_{T_{3,4}}(t) + \Upsilon_{T_{3,2}}(t)\\
=& \Upsilon_{T_{8,9}}(t) + 2\cdot \Upsilon_{T_{3,4}}(t) + \Upsilon_{T_{2,3}}(t).
\end{align*}
\end{example}

As a consequence of calculating $\Upsilon_{T_{p,q}}$, one finds for example the following. \begin{Obs}
For a torus knot $T=T_{p,q}$, with $p<q$,  one has
\[\U_T(t)=-t\tau(T)=-tg_4(T) \for t\leq\frac{2}{p}\]
and
\[\U_T(t)\geq-tg_4(T)+\left\lfloor\frac{q}{p}\right\rfloor(p t-2)>-tg_4(T) \for \frac{2}{p}<t\leq 1.\]

\end{Obs}
However, it is the point of this paper that this follows  from Observation~\ref{Obs:UtforTpnp1} without actually having calculated $\Upsilon$  (and for a much larger class of knots including the knots $L$ in Theorem~\ref{thm:unknottingtotorusknotandbraidindex}); compare Proposition~\ref{prop:qpbraidindexdetection}.

We note that Wang~\cite{Wang} independently calculated that
\[\U_T(t)=-t\tau(T)=-tg_4(T) \for t\leq\frac{2}{p}\]
and
\[\U_T(t)>-tg_4(T) \for \frac{2}{p}<t\leq 1,\]
for torus knots and more generally algebraic knots of multiplicity (which equals the braid index) $p$.

\section{$\Upsilon$ for quasi positive knots and proofs of Theorems~\ref{thm:unknottingtotorusknotandbraidindex} and~\ref{thm:concbraidindex}}\label{sec:upstorusknotsandmainresult}

In this Section, we use the calculation of $\Upsilon$ for torus knots provided in Section~\ref{sec:Upsfortorusknots}
to obstruct the existence of cobordisms between knots that `contain' full-twists and knots that are `contained in' full-twists.
\subsection{Knots with optimal cobordisms to torus knots}
We first make `containing a full-twist' precise.
\begin{Definition}
For a knot $K$ we denote by $n_K$ the largest positive integer such that $d(K,T_{n_K,n_K+1})=g_4(K)-g_4(T_{n_K,n_K+1})$. Furthermore, we denote by $k_K$ the largest positive integer such that $d(K,T_{n_K,k_Kn_K+1})=g_4(K)-g_4(T_{n_K,k_Kn_K+1})$. If $n_K=1$, we define $k_K$ to be 0.
\end{Definition}
The notion of $n_K$ and $k_K$ is motivated by the study of knots that are closures of $n$ stranded braids $\beta$ such that $\beta=(\Delta^2)^k\alpha$, where $\alpha$ is a positive $n$ strand braid and $k$ is a positive integer. In other words, $n_K$ is meant to capture a notion of `maximal positive full-twist contained in $K$' up to concordance. Note that $T_{1,k}$ is the unknot for any $k$, so some additional convention is necessary to define $k_K$ when $n_K=1$; the particular choice of convention will have no bearing on what follows. In the opposite direction, we try to capture the notion of `the smallest number $m$ such that a knot $K$ is contained in a torus knot of braid index $m$'.
\begin{Definition}
For a knot $K$ we denote by $m_K$ the smallest positive integer such that there exists a positive integer $k$ such that $d(K,T_{m_K,km_K+1})=g_4(T_{m_K,km_K+1})-g_4(K)$. If no such positive integer exists we set $m_K=\infty$.
\end{Definition}
The following Theorem can be seen as an improvement on the triangle inequality given in~\eqref{eq:tri}.
\begin{theorem}\label{thm:triimprovmentforknotscontainingtwists}
For all knots $K$ and $L$, we have that
\[k_L(n_L-m_K)+g_4(K)-g_4(L)\leq d(K,L)
.\]
\end{theorem}
Before we provide a proof, we note that
Theorem~\ref{thm:triimprovmentforknotscontainingtwists} implies Theorem~\ref{thm:unknottingtotorusknotandbraidindex} by the following two lemmata.

For quasi-positive knots and (more generally) knots for which the slice-Bennequin inequality is sharp, $m_K$ is bounded by the braid index:
\begin{lemma}~\label{lem:m_Kforquasipos} Let $L$ be a knot that arises as the closure of an $n$-braid $\beta$ for which the slice-Bennequin inequality is an equality; i.e.
\begin{equation}\label{eq:sliceBennequinissharp}
g_4(L)=\frac{\l(\beta)-n+1}{2}.\end{equation}
Then the braid index of $L$ is larger than or equal to $m_L$.
\end{lemma}
Here the \emph{algebraic length} or \emph{writhe} $l(\beta)$ of an $n$-braid $\beta$ is given by the exponent homomorphism $l$ defined from the braid group on $n$ strands to $\Z$ by mapping the positive generators $a_i$ to $1$.
\begin{lemma}\label{lem:n_Kandt_Kforquasipos}
If a knot $K$ is the closure of an $n$-braid $\beta$ of the form $(a_1\cdots a_{n-1})^{nk+1}\alpha$, for positive integers $n$ and $k$ and a quasi-positive $n$-braid $\alpha$, then \[n_K=m_K=n\et k_K\geq k.\]
The same is true if $\beta=(\Delta^2)^k\alpha$ for some positive braid $\alpha$.
\end{lemma}

To prove Lemma~\ref{lem:m_Kforquasipos} we invoke the generalized Jones conjecture~\cite{DynnikovPrasolov_13,LaFountainMenasco_14}, which states that if a knot $K$ with braid index $b$ arises as the closure of an $n$-braid $\beta$ and a $b$-braid $\beta'$, then \begin{equation}\label{eq:genJonesConj}|\l(\beta)-\l(\beta')|\leq n-b.\end{equation} In fact, we use the following 
consequence of the generalized Jones conjecture.
\begin{lemma}\label{lem:Jonesimplies} Let $K$ be a knot and denote its braid index by $b$.
If $K$ arises as the closure of an $n$-braid $\beta$ for which the slice-Bennequin inequality is an equality, i.e.
\[g_4(K)=\frac{\l(\beta)-n+1}{2},\]
then the same is true for all $b$-braids $\beta'$ with closure $K$; i.e.
\[
g_4(K)=\frac{\l(\beta')-b+1}{2}.
\]
\end{lemma}
\begin{proof}[Proof of Lemma~\ref{lem:Jonesimplies}]
We calculate
\[g_4(K)=\frac{\l(\beta)-n+1}{2}\leq \frac{\l(\beta')-b+1}{2}\leq g_4(K),\]
where the first inequality invokes the generalized Jones conjecture~\eqref{eq:genJonesConj} and the second inequality is the slice-Bennequin inequality~\cite{rudolph_QPasObstruction} for the braid $\beta'$.\end{proof}

In the arguments that follow, we will use cobordisms between knots constructed using an even number of \emph{band moves} or \emph{saddle moves}---a cobordism given by one saddle guided by an embedded arc in $S^3$ starting and ending on the knot (or link) in question. 
Figure~\ref{fig:saddlemoveandcrossingchange} provides a diagrammatic description of a saddle move.
\begin{figure}[h]
\centering
\def\svgscale{1.9}
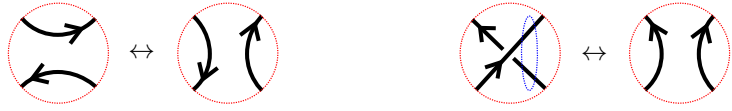
\caption{Diagrammatic representations of a saddle move (left) and a smoothing of a crossing (right). For the smoothing of a crossing we indicate (blue) of how to view it as a saddle move.}
\label{fig:saddlemoveandcrossingchange}
\end{figure}
In fact, most of our cobordisms will arise from saddle moves that are most easily seen as smoothings of crossings in a diagram for the knot in question; compare Figure~\ref{fig:saddlemoveandcrossingchange}.

\begin{proof}[Proof of Lemma~\ref{lem:m_Kforquasipos}]
Let $b$ be the braid index of $L$ and let $\beta'$ be a $b$-braid with closure $L$.
Let $\alpha$ be a positive $b$-braid such that $\beta'\alpha$ has closure $T_{b,kb+1}$ for some positive integer $k$. In particular, there exists a cobordism $C$ between $L$ and $T_{b,kb+1}$ of genus
\begin{equation}\label{eq:genuscobanybraidtotorusknot}\frac{\l(\alpha)}{2}=g_4(T_{b,kb+1})-\frac{\l(\beta')-b+1}{2}.
\end{equation}
Indeed, the braid $\beta'$ can be obtained from $\beta'\alpha$ by deleting $\l(\alpha)$ generators $a_i$; and so, since deleting a generator can be realized in a braid diagram by smoothing a crossing, $C$ can be taken to be the cobordism given by smoothing the $\l(\alpha)$ crossings corresponding to the generators of $\alpha$.
Since $L$ arises as the closure of a braid for which the slice--Bennequin inequality is an equality,
Lemma~\ref{lem:Jonesimplies} implies that \[g_4(L)=\frac{\l(\beta')-b+1}{2}.\]
Hence by~\eqref{eq:genuscobanybraidtotorusknot},
the genus of $C$ is $g_4(T_{b,kb+1})-g_4(L)$. Since $d(T_{b,kb+1},L)\leq g(C)$, and the triangle inequality \eqref{eq:tri} says that \[ d(T_{b,kb+1},L) \geq g_4(T_{b,kb+1})-g_4(L),\] it follows that \[d(T_{b,kb+1},L) = g(C) = g_4(T_{b,kb+1})-g_4(L) ,\] and so $b\geq m_L$.
\end{proof}

\begin{proof}[Proof of Lemma~\ref{lem:n_Kandt_Kforquasipos}]
We prove $n_K=m_K=n$ by establishing $n\geq m_K$, $m_K\geq n_K$, and $n_K\geq n$.
First, we observe that Lemma~\ref{lem:m_Kforquasipos} implies $n\geq m_K$.
Next, we show $m_K\geq n_K$. Indeed, Theorem~\ref{thm:triimprovmentforknotscontainingtwists} (where $L$ is chosen to be $K$) yields
\[k_K(n_K-m_K)\leq d(K,K)=0,\]
and, therefore, $m_K\geq n_K$ (since $k_K$ is a positive integer).
Finally, we show $n_K\geq n$.
If $\beta=(a_1\cdots a_{n-1})^{nk+1}\alpha$ for some quasi-positive braid $\alpha$, then $\alpha$ can be written as a product of $\l(\alpha)$ conjugates of generators $a_i$ by the definition of quasi-positivity. Smoothing all these $a_i$ in the $\alpha$ part of $(a_1\cdots a_{n-1})^{nk+1}\alpha$ yields a cobordism $C$ between $K$ and $T_{n,{nk+1}}$ of genus
\[\frac{\l(\alpha)}{2}=\frac{(n-1)nk+\l(\alpha)}{2}-\frac{(n-1)nk}{2}=g_4(K)-g_4(T_{n,{nk+1}}).\]
If $\beta=(\Delta^2)^k\alpha$ for some positive braid $\alpha$, then smoothing all but one $a_i$ in the $\alpha$ part of $(\Delta^2)^k\alpha$ for all $1\leq i\leq n-1$ yields a cobordism between $K$ and $T_{n,nk+1}$ of genus
\[\frac{\l(\alpha)-(n-1)}{2}=
g_4(K)-g_4(T_{n,{nk+1}}).\]
Therefore,
\begin{equation}\label{eq:d=g-g}d(K,T_{n,nk+1})
=g_4(K)-g_4(T_{n,nk+1})
.\end{equation}
In particular, \begin{align*}d(K,T_{n,n+1})&\leq d(K,T_{n,nk+1}) + d(T_{n,nk+1},T_{n,n+1})\\&= g_4(K)-g_4(T_{n,nk+1})+g_4(T_{n,nk+1})-g_4(T_{n,n+1})\\&=g_4(K)-g_4(T_{n,n+1})\leq d(K,T_{n,n+1}),\end{align*} and so $n_K\geq n$.

We conclude the proof by noting that $k_K\geq k$ follows from~\eqref{eq:d=g-g}, $n=n_K$, and the definition of $k_K$.
\end{proof}
\subsection{Proof of Theorem~\ref{thm:triimprovmentforknotscontainingtwists}}
We use the calculation of $\Upsilon$ for torus knots (in fact only Observation~\ref{Obs:UtforTpnp1}) to deduce the following.
\begin{prop}\label{prop:qpbraidindexdetection}
For all knots $K$
, we have
\[\U_K(t)=-tg_4(K)=-t\tau(K) \for t\leq\frac{2}{m_K} \et\]
\begin{equation}
\label{eq:upsilonnotslice}
\U_K(t)\geq-tg_4(K)+k_K(n_K t-2)>-tg_4(K) \for \frac{2}{n_K}<t\leq 1.\end{equation}
\end{prop}

Note that it is not true in general that $\Upsilon$ detects the slice genus, nor that $g_4(K)=\tau(K)$. Part of the claim of Propostion \ref{prop:qpbraidindexdetection} is that these are both true for all $K$ for which $m_K < \infty$.

In particular, if $K$ is a the closure of a quasi-positive $m$-braid, then $\Upsilon_K(t)=-t\tau(K)=-tg_4(K)$ for $t\leq \frac{2}{m}$.
As a consequence (compare~\cite[Lemma~4]{rudolph_QPasObstruction}), for $t\leq\frac{2}{m}$, $-\frac{\Upsilon(t)}{t}$ satisfies the slice-Bennequin inequality
\begin{equation}\label{eq:Bennequin-type}-\frac{\Upsilon_{\widehat{\beta}}(t)}{t}\geq \frac{\l(\beta)-m+1}{2}\quad\text{ for all }m\text{-braids }\beta;\end{equation}
while, for $\frac{2}{m}<t\leq 1$, $-\frac{\Upsilon(t)}{t}$ fails to satisfy such an inequality.
The fact that whether or not $\Upsilon(t)$ (for a fixed $t$) satisfies a slice-Bennequin inequality depends on the number of strands should not be seen as a draw back but as a feature! Indeed, we use this to prove Theorems~\ref{thm:unknottingtotorusknotandbraidindex} and~\ref{thm:triimprovmentforknotscontainingtwists}.


\begin{proof}[Proof of Proposition~\ref{prop:qpbraidindexdetection}]
For the proof of the first part we fix $t\in(0,\frac{2}{m_K}]$. 
The quantity $\frac{\Upsilon(\cdot)}{t}$ is a concordance invariant, additive on connected sums, and its absolute value is a lower bound for the slice genus~\cite{OSS_2014}.
Therefore, one has
\[\Upsilon_K(t){\geq} -tg_4(K)\et td(K,T_{m_K,km_K+1})\geq \Upsilon_K(t)-\Upsilon_{T_{m_K,km_K+1}}(t).\]

Let $k$ be a positive integer such that $d(K, T_{m_K,km_K+1})=g_4(T_{m_K,km_K+1})-g_4(K)$, which exists by the definition of $m_k$.
We calculate
\begin{align*}
\Upsilon_K(t)\geq -tg_4(K)&=-tg_4(K)+tg_4(T_{m_K,km_K+1})-tg_4(T_{m_K,km_K+1})\\&=td(K,T_{m_K,km_K+1})-tg_4(T_{m_K,km_K+1})
\\&=td(K,T_{m_K,km_K+1})+\Upsilon_{T_{m_K,km_K+1}}(t)\\&\geq
\Upsilon_K(t)-\Upsilon_{T_{m_K,km_K+1}}(t)+\Upsilon_{T_{m_K,km_K+1}}(t)=\Upsilon_K(t),\end{align*}
where the third equality uses the first part of Observation~\ref{Obs:UtforTpnp1}. Thus, we have
$\Upsilon_K(t)=-tg_4(K)$. 

For the proof of the second part we fix $t\in(\frac{2}{n_K},1]$. By the definition of $n_K$ and $k_K$, there exists a cobordism $C$ between $K$ and the torus knot $T=T_{n_K,k_Kn_K+1}$ of genus
\[g(C)=d(K,T)=g_4(T\sharp m({K}))=g_4(K)-g_4(T)
.\]
Combining the second part of Observation~\ref{Obs:UtforTpnp1} with the fact that $|\frac{\Upsilon}{t}|$ is a lower bound for the slice genus, we find
\begin{align*}\Upsilon_K(t)+tg_4(T)-k_K(tn_K-2)&\geq\Upsilon_K(t)-\Upsilon_{T}(t)
=-\Upsilon_{T\sharp m({K})}(t)\\&\geq-tg_4(T\sharp m({K}))=-tg_4(K)+tg_4(T),\end{align*}
which in turn yields
\[\U_K(t)\geq-tg_4(K)+k_K(n_K t-2)>-tg_4(K).\]
\end{proof}

With Proposition~\ref{prop:qpbraidindexdetection}, the proof of Theorem~\ref{thm:triimprovmentforknotscontainingtwists} is immediate.
\begin{proof}[Proof of Theorem~\ref{thm:triimprovmentforknotscontainingtwists}]
We only consider the case where $m_K<n_L$, since otherwise the statement of Theorem~\ref{thm:triimprovmentforknotscontainingtwists} is contained in the triangle inequality~\eqref{eq:tri}.
Using that $|\frac{\Upsilon(t)}{t}|$ is a bound for the slice genus, we have
\[d(K,L)=g_4(L\sharp m({K}))\geq\frac{m_K}{2}\Upsilon_{L\sharp m({K})}(\frac{2}{m_K})
=\frac{m_K}{2}\Upsilon_L(\frac{2}{m_K})-\frac{m_K}{2}\Upsilon_K(\frac{2}{m_K}).\]
Proposition~\ref{prop:qpbraidindexdetection} yields
\[\frac{m_K}{2}\Upsilon_L(\frac{2}{m_K})-\frac{m_K}{2}\Upsilon_K(\frac{2}{m_K})\geq
-\frac{m_K}{2}(\frac{2}{m_K}g_4(L)-k_L(n_L\frac{2}{m_K}-2))+\frac{m_K}{2}\frac{2}{m_K}g_4(K),\]
and, therefore,
\[d(K,L)\geq -g_4(L)+k_L(n_L-m_K)+g_4(K).\]
\end{proof}
\subsection{Proof of Theorem~\ref{thm:concbraidindex}}
The proof of Theorem~\ref{thm:concbraidindex} provided below can be immediately adapted to yield
\begin{prop}\label{prop:concbraidindex}
Let $L$ be a knot. If $k_L\geq n_L-1$, then all knots $K$ that are concordant to $L$ have braid index at least $n_L$. \qed
\end{prop}
First note that Equation \eqref{eq:upsilonnotslice}, to which we will refer in the proof below, only gives a nontrivial restriction if $n_L>2$, since $\Upsilon$ is only defined on $[0,2]$. The cases $n_L=1,2$ require separate attention. If $n_L=1$, the statement of the Theorem is vacuous. If $n_L=2$, then $g_4(L) \geq g_4(T_{2,2k+1}) = k >0$. Hence $L$ is not slice, so not concordant to any 1-braid.
\begin{proof}[Proof of Theorem~\ref{thm:concbraidindex}]
Having dispensed with the cases $n=1,2$, we turn to the general case where $n>2$. Assume towards a contradiction that there is an $(n-1)$-braid $\gamma$ whose closure $\widehat{\gamma}$ is concordant to $L$. By applying 
the slice-Bennequin 
inequality~\eqref{eq:Bennequin-type} to $\gamma$, we have
\begin{equation}\label{eq:sbineq}
 -\frac{\Upsilon_{\widehat{\gamma}}(t)}{t}
 \geq \frac{\ell(\gamma) - ((n-1)-1)}{2}
 \end{equation} for $t\leq \frac{2}{n-1}$.

Applying the slice-Bennequin 
inequality for $\tau$ established by Livingston~\cite[Corollary~11]{Livingston_Comp}, i.e.
\[\tau(\widehat{\beta})\geq\frac{\l(\beta)-(n-1-1)}{2}\quad\text{for all } (n-1)\text{-braids},\]  to $m\gamma$---the mirror image of $\gamma$---yields
 \[\tau(\widehat{m\gamma})\geq \frac{\ell(m\gamma) - ((n-1)-1)}{2}.\]
 Using
$\tau(\widehat{\gamma})=-\tau(\widehat{m\gamma})$ and $\ell(m\gamma) =-\ell(\gamma)$ this yields
\[ \tau(\widehat{\gamma})\leq \frac{\ell(\gamma)+(n-2)}{2}.\] Combined with \eqref{eq:sbineq}, this gives \[ -\frac{\Upsilon_{\widehat{\gamma}}(t)}{t} \geq \tau(\widehat{\gamma})-(n-2),\] for $t\leq \frac{2}{n-1}$. In particular, at $t=\frac{2}{n-1}$, we get \begin{equation}\label{eq:prop19b}
-\Upsilon_{\widehat{\gamma}}\left( \frac{2}{n-1}\right) \geq \frac{2}{n-1}\tau(\widehat{\gamma})-2\frac{n-2}{n-1}.
\end{equation} On the other hand, for $L$, Proposition \ref{prop:qpbraidindexdetection} tells us that \begin{equation}\label{eq:prop19a}
-\Upsilon_L\left( \frac{2}{n-1}\right) \leq \frac{2}{n-1}\tau(L) -nk\left( \frac{2}{n-1}-\frac{2}{n}\right),
\end{equation} since $\frac{2}{n-1}>\frac{2}{n}$. Since $\tau$ and $\Upsilon$ are concordance invariants, combining \eqref{eq:prop19b} and \eqref{eq:prop19a} gives \[ 2\frac{n-2}{n-1} \geq \frac{2k}{n-1},\] contradicting the fact that $\beta$ has $k>n-2$ twists.
\end{proof}

\section{Homogenization and comparison between $\Ut$ and the Levine-Tristram signature $\so$.}\label{sec:homogenization}
Given a link invariant $\tau$ one can define a braid invariant (also denoted by $\tau$) by setting \[\tau(\beta)=\tau\left(\widehat{\beta}\right),\]
where $\widehat{\beta}$ denotes the (standard) closure of a braid $\beta$.

Fix a positive integer $n$ and study the braid group $B_n$ on $n$-strands. Assume $\tau$ takes values in $\R$ and that $\tau\colon B_n\to\R$ is a quasi-morphism; that is, 
there exists a positive real $d$, called the \emph{defect}, such that
\[|\tau(\alpha\beta)-\tau(\alpha)-\tau(\beta)|\leq d\] for all \(\alpha, \beta \) in \(B_n.\)
Then the homogenized invariant
\[
\widetilde{\tau}\colon B_n\to \R,\quad \beta\mapsto \lim_{k\to \infty}\frac{\tau(\beta^k)}{k}\]
is well-defined.
For example, Gambaudo and Ghys~\cite{GambaudoGhys_BraidsSignatures} studied the homogenization $\So$ of the Levine-Tristram signatures $\so$, introduced by Levine and Tristram~\cite{levine,tristram} for unit complex numbers $\omega$ as a generalization of Trotter's classical signature $\s=\s_{-1}$~\cite{Trotter_62_HomologywithApptoKnotTheory}.

In this section,
we compare the homogenizations of $\Upsilon(t)$ and $\sigma_{e^{\pi i t}}$. 
We show that for the standard braids representing torus links one has that the homogenization of $\Upsilon(t)$ equals $\frac{\widetilde{\sigma_{e^{\pi i t}}}}{2}$ and that for 3-braids the homogenization of $\Upsilon(1)$ equals $\frac{\widetilde{\sigma_{e^{\pi i}}}}{2}=\frac{\widetilde{\sigma}}{2}$
The latter leads to examples of positive $3$-braids that have non-convex $\Upsilon$.

\subsection{Definition and general properties of homogenizations}
We start with defining the homogenization $\Upsilon$ and recalling some elementary properties. Since $\Upsilon$ is a knot invariant not defined on links, some extra care is in order when defining its homogenization.
We set the homogenization of $\Upsilon$ to be \begin{equation}\label{eq:defHomUpsilon}\widetilde{\Upsilon(t)}(\beta)=\lim_{k\to\infty}\frac{\Upsilon_{\widehat{\beta^{k(n!)}a_1a_2\cdots a_{n-2}a_{n-1}}}(t)}{(n!)k}.\end{equation}
Morally, this is the homogenization
\[
\lim_{k\to\infty}\frac{\Upsilon_{\widehat{\beta^{k}}}(t)}{k};\]
however, because the link $\widehat{\beta^{k}}$ might have more than one component 
 and $\Upsilon(t)$ is only defined for knots, the former definition is the sensible one. Indeed, by taking powers that are multiples of $n!$ we guarantee that $\beta^{k(n!)}$ is a pure braid and, thus, the closure of $\beta^{k(n!)}a_1a_2\cdots a_{n-2}a_{n-1}$ is a knot.
 Below, when we write $\lim_{k\to\infty}$, we will always mean that the limit is taken over $k$ that are a multiples of $n!$. This is further justified by the third item in the next remark.
 \begin{Remark}\label{rem}
 The homogenization of $\Upsilon$ behaves well, by general principles that hold for homogenizations coming from knot invariants with the property that their absolute value is a lower bound for slice genus; compare with Brandenbursky's work~\cite{Brandenbursky_11}. In particular, we have the following.
 \begin{itemize}
\item The homogenization $\HUt$ is well-defined; i.e.~the limit exists.

\item The choice of the braid $(a_1a_2\cdots a_{n-2}a_{n-1})$ is not relevant; i.e.~any other braid that closes to a knot yields the same homogenization.

\item Rather than letting the limit run over multiples of $n!$, we could take the limit over any other sequence of integers $(b_k)_{k\in\N}$ going to $\infty$ for which the closures of the $\beta^{b_k}$ are pure braids.

\item If $\widehat{\beta}$ is a knot, then $\HUt(\beta)-\Upsilon_{\widehat{\beta}}(t)\leq t\frac{n-1}{2}$.

\item The homogenization $\HUt\colon B_n\to \R$ is a quasi-morphism with defect $d=t({n-1})$.
\end{itemize}\end{Remark}

\subsection{The homogenization of $\Upsilon$ for small $t$ and standard torus braids}
In this subsection, we determine $\HUt\colon B_n\to \R$ for $t\leq\frac{2}{n}$ and we calculate $\HUt$ on the standard torus braids $(a_1a_2\cdots a_{n-1})^m\in B_n$. These are immediate consequences of our $\Upsilon$ calculations above.

As a corollary of Proposition~\ref{prop:qpbraidindexdetection} we have:
\begin{corollary}\label{cor:u=sfortsmall}Let $\beta$ be an $n$-braid of algebraic length $\l(\beta)$. 
One has
$\HUt(\beta)=-t\frac{\l(\beta)}{2}$ for $t\leq\frac{2}{n}$.
In other words, $\HUt(\beta)=\frac{\widetilde{\sigma_{e^{\pi i t}}}(\beta)}{2}$ for $t\leq\frac{2}{n}$.
\end{corollary}
\begin{proof}
We first consider the case where $\beta$ is a positive (or quasi-positive) braid. Thus, by Lemma~\ref{lem:m_Kforquasipos}, we have $n\geq m_{\widehat{\beta^ka_1\cdots a_{n-1}}}$ for all positive integers $k$ such that the closure of $\beta^ka_1\cdots a_{n-1}$ is a knot. For all $t\leq\frac{2}{n}$, we therefore find
\begin{align*}\HUt(\beta)&=\lim_{k\to\infty}\frac{\Upsilon_{\widehat{\beta^ka_1\cdots a_{n-1}}}(t)}{k}
\\&=\lim_{k\to\infty}\frac{-tg_4\left(\widehat{\beta^ka_1\cdots a_{n-1}}\right)}{k}
\\&=\lim_{k\to\infty}\frac{-tl(\beta^k)}{2k}=\frac{-tl(\beta)}{2},\end{align*}
where Proposition~\ref{prop:qpbraidindexdetection} is used in the second equality.

Otherwise,
write $\beta=\alpha(\Delta^{2})^{-l}$, where $l$ is a positive integer, $\alpha$ is a positive $n$-braid, and $\Delta^{2}$ denotes the positive full-twist on $n$-strands.
Since $\Delta^{2}$ commutes with every other $n$-braid, we have $\beta^k=\alpha^k(\Delta^{2})^{-kl}$ for any positive integer $k$ and so
\begin{equation}\label{eq}\left|\HUt(\beta^k)-\HUt(\alpha^k)-\HUt((\Delta^{2})^{-kl})\right|\overset{\text{Remark~\ref{rem}}}{\leq} t(n-1)\for\all k\in\Z.\end{equation}
Using~\eqref{eq} and the above calculation for positive braids, we find
\begin{align*}\HUt(\beta)&=
\lim_{k\to\infty}\frac{\HUt(\beta^k)}{k}
\\&\overset{\eqref{eq}}=
\lim_{k\to\infty}\frac{\HUt(\alpha^k)+\HUt((\Delta^{2})^{-kl})}{k}\\&=
\lim_{k\to\infty}\frac{\HUt(\alpha^k)-\HUt((\Delta^{2})^{kl})}{k}
\\&=
\lim_{k\to\infty}\frac{\HUt(\alpha^k)}{k}-\lim_{k\to\infty}\frac{\HUt((\Delta^{2})^{kl})}{k}\\
&
=\HUt(\alpha)-\HUt((\Delta^{2})^{l})
\\&
=\frac{-tl(\alpha)}{2}-\frac{-tl((\Delta^{2})^{l})}{2}
\\&=-\frac{tl(\alpha(\Delta^{2})^{-l})}{2}=-\frac{tl(\beta)}{2}
.\end{align*}

The equality
$\HUt(\beta)=\frac{\widetilde{\sigma_{e^{\pi i t}}}(\beta)}{2}$ follows, since
$\widetilde{\sigma_{e^{\pi i t}}}(\beta)=-t{\l(\beta)}$  for $t\leq\frac{2}{n}$, as mentioned in~\cite{GambaudoGhys_BraidsSignatures}.
\end{proof}

In case of the standard torus link braids, more can be said.
As a consequence of Corollary~\ref{cor:upsfortnnk+1} and the formula
\begin{equation}\label{eq:homSigfortorusbraids}\widetilde{\sigma_{e^{\pi i t}}}((a_1a_2\cdots a_{n-1})^m)=\frac{2m}{n}\left(-(i+1)i-\frac{1}{2}n(n-1-2i)t\right)\for t\in[\frac{2i}{n},\frac{2i+2}{n}],\end{equation}
given in~\cite[Proposition~5.2]{GambaudoGhys_BraidsSignatures}, one finds:
\begin{corollary}\label{cor:u=sfortoruslinks}
For all $t\in[0,2]$, integers $m$, and positive integers $n$, one has $\HUt=\frac{\widetilde{\sigma_{e^{\pi i t}}}}{2}$ on 
the $n$-braid $(a_1a_2\cdots a_{n-2}a_{n-1})^m$. 
\end{corollary}
\begin{proof}For all integers $k$, we write
\[((a_1a_2\cdots a_{n-2}a_{n-1})^m)^k=(a_1a_2\cdots a_{n-2}a_{n-1})^{n\lfloor\frac{mk}{n}\rfloor}(a_1a_2\cdots a_{n-2}a_{n-1})^{mk-n\lfloor\frac{mk}{n}\rfloor}.\] By Remark~\ref{rem} and the fact that $mk-n\lfloor\frac{mk}{n}\rfloor\leq n$, this yields

\begin{equation}\label{eq2}\left|\HUt(((a_1a_2\cdots a_{n-2}a_{n-1})^m)^k)-\HUt((a_1a_2\cdots a_{n-2}a_{n-1})^{n\lfloor\frac{mk}{n}\rfloor})
\right|
{\leq} C(n), 
\end{equation}
where $C(n)$ is a constant only depending on $n$.
We calculate
\begin{align*}\HUt((a_1a_2\cdots a_{n-2}a_{n-1})^m)&=
\lim_{k\to\infty}\frac{\HUt(((a_1a_2\cdots a_{n-2}a_{n-1})^m)^k)}{k}\\
&=\lim_{k\to\infty}\frac{\HUt((a_1a_2\cdots a_{n-2}a_{n-1})^{n\lfloor\frac{mk}{n}\rfloor})}{k}\\
&=\lim_{k\to\infty}
\frac{\lim_{k'\to\infty}\frac{\Upsilon_{T_{n,k'n\lfloor\frac{mk}{n}\rfloor+1}}(t)}{k'}}{k}\\
&=\lim_{k\to\infty}
\frac{\lim_{k'\to\infty}\frac{k'\lfloor\frac{mk}{n}\rfloor\Upsilon_{T_{n,n+1}}(t)}{k'}}{k}\\
&=\lim_{k\to\infty}
\frac{\lfloor\frac{mk}{n}\rfloor\Upsilon_{T_{n,n+1}}(t)}{k}\\
&=\frac{m}{n}\Upsilon_{T_{n,n+1}}(t)\\
&=\frac{m}{n}\left(-(i+1)i-\frac{1}{2}n(n-1-2i)t\right)\\
&=\frac{\widetilde{\sigma_{e^{\pi i t}}}((a_1a_2\cdots a_{n-2}a_{n-1})^m)}{2},
\end{align*}
where~\eqref{eq2}, Proposition~\ref{prop:upsalltorus}, Proposition~\ref{prop:OSStorus}, and~\eqref{eq:homSigfortorusbraids} are used in the second, fourth, second to last, and last equality, respectively.
\end{proof}
\subsection{Comparison between $\Upsilon(1)$ and $\frac{\sigma}{2}$}
The study of the difference of $\upsilon=\Upsilon(1)$ and $\frac{\sigma}{2}=\frac{\s_{-1}}{2}$ is of particular interest, since this difference provides a lower bound for the smooth four-dimensional
crosscap number~\cite[Theorem~1.2]{OSS_2015}. 
\begin{prop}\label{prop:sig=ups}
We have that $\widetilde{\u}=\widetilde{\frac{\sigma}{2}}$ for all $3$-braids. In particular, \[\left|\Upsilon_K(1)-\frac{\sigma(K)}{2}\right|\leq 2\] for all knots $K$ that arise as closures of $3$-braids.
\end{prop}
The second part of Proposition~\ref{prop:sig=ups} follows from the first part and Remark~\ref{rem}. It is worth noting that for $t\neq1$, the analog of Proposition~\ref{prop:sig=ups} is false; see Example~\ref{Ex:a^3b^7} below. Proposition~\ref{prop:sig=ups} leads
to examples of knots that arise as closures of positive $3$-braids for which $\Upsilon$ is non-convex. This answers a question of Borodzik and Hedden in the negative~\cite[Question~1.5]{BorodzikHedden}. We provide these examples in detail before proving Proposition~\ref{prop:sig=ups}.

In Example~\ref{Ex:a1a1a2a2} and the proof of Proposition~\ref{prop:sig=ups}, we will need the values of $\Upsilon$ of torus knots of braid index at most 3. 
Recall, that for torus knots of braid index 2, one has $\Upsilon(t)=-\tau(t)$ for all $t\leq 1$ (for example by Proposition~\ref{prop:OSStorus} or by the fact that $\Upsilon$ is linear on $[0,1]$ for alternating knots~\cite[Theorem~1.14]{OSS_2014}). For torus knots of braid index 3, $\Upsilon$ is given as follows.
For all positive integers $n$, we have
\begin{equation}\label{eq:upsilonforindex3torusknots}
\upsilon(T_{3,3n+1})=\upsilon(T_{3,3n+2})+1=-2n \et \upsilon(T_{3,-3n-1})=\upsilon(T_{3,-3n-2})-1=2n
.\end{equation} Equation~\eqref{eq:upsilonforindex3torusknots} can be calculated using~\cite[Theorem~15]{OSS_2014}. See for example~\cite[Proposition~28]{Feller_15_MinCobBetweenTorusknots} or use Proposition~\ref{prop:upsalltorus}.
This determines $\Upsilon$ for torus knots of braid index 3, since
\[\Upsilon_{T_{3,p}}(t)=-\tau(T_{3,p})t\quad\text{for }t\leq\frac{2}{3}\]
and $\Upsilon_{T_{3,p}}(t)$ is linear on $[\frac{2}{3},1]$.

\begin{Example}\label{Ex:a1a1a2a2}
Let $\beta_n$ be the 3-braid $(a_1^2a_2^2)^n$. By Corollary~\ref{cor:u=sfortsmall}, we have
\[\HUt(\beta_n)=-2tn\quad\text{for } 0\leq t\leq \frac{2}{3}.\]
The asymptotic signature $\widetilde{\frac{\sigma}{2}}(\beta_n)$ is $-n$; see e.g.~\cite{Stoimenow_08,Feller_15_ASharpSignatureBoundForPositiveFour-Braids}.
Therefore, we have that $\widetilde{\u}$ is $-n$ by Proposition~\ref{prop:sig=ups}. In particular, $\HUt(\beta_n)$ is not convex as a function of $t$, since \[\widetilde{\Upsilon(\tfrac{4}{3})}(\beta_n)=\widetilde{\Upsilon(\tfrac{2}{3})}(\beta_n)=-2n\frac{2}{3}<-n=\widetilde{\Upsilon(1)}(\beta_n).\]
As a consequence, the knots $K_{n}$ obtained as the closure of $\beta_n(a_1a_2)$, where $n$ is a positive integer, have non-convex $\Upsilon$ for large enough $n$.
We provide the calculation that establishes the latter statement without reference to Proposition~\ref{prop:sig=ups},
since in the proof of Proposition~\ref{prop:sig=ups} we use part of this calculation.
In fact, we prove that $\U_{K_{n}}$ is not convex for $n\geq 6$:

We start by observing that there is a genus 1 cobordism between $\widehat{\beta_n(a_1a_2)}\sharp T_{2,2n+1}$  and $T_{3,3n+1}$ or, in other words, $g_4(K_n\sharp T_{2,2n+1}\sharp m({T_{3,3n+1}}))\leq 1$. For this, we note that $\beta_n(a_1a_2)a_2^{2n}$ equals $(a_1a_2)^{3n+1}$ as 3-braids. Thus,
$T_{3,3n+1}$ is the closure of $\beta_n(a_1a_2)a_2^{2n}$. A genus 1 cobordism from $\widehat{\beta_n(a_1a_2)}\sharp T_{2,2n+1}$ to $\widehat{\beta_n(a_1a_2)a_2^{2n}}=T_{3,3n+1}$ is indicated in Figure~\ref{fig:genus1cob}.
\begin{figure}[h]
\centering
\def\svgscale{1.8}
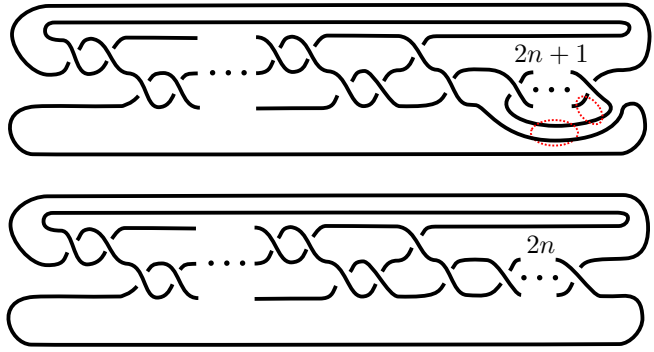
\caption{The knot $\widehat{\beta_n(a_1a_2)}\sharp T_{2,2n+1}$ (top) with $2$ spheres (red) indicating how $2$ saddle moves, which correspond to a genus 1 cobordism, yield $\widehat{\beta_n(a_1a_2)a_2^{2n}}=T_{3,3n+1}$ (bottom).
}
\label{fig:genus1cob}
\end{figure}

Therefore, we have
\[\left|\Upsilon_{T_{3,3n+1}}(t)-\Upsilon_{T_{2,2n+1}\sharp K_{n}}(t)\right|=\left|\Upsilon_{K_n\sharp T_{2,2n+1}\sharp m({T_{3,3n+1}})}(t)\right|\leq 1-|1-t|\]
since $\left|\frac{\Upsilon}{1-|1-t|}\right|$ is a lower bound for the slice genus. We rewrite this as
\begin{align}
\label{eq:lowerboundUps(Kn)}\Upsilon_{T_{3,3n+1}}(t)-\Upsilon_{T_{2,2n+1}}(t) - (1-|1-t|)
&\leq\Upsilon_{K_{n}}(t)\et\\
\label{eq:upperboundUps(Kn)}\Upsilon_{T_{3,3n+1}}(t)-\Upsilon_{T_{2,2n+1}}(t) + ( 1-|1-t|)&\geq \Upsilon_{K_{n}}(t).\end{align}
Non-convexity follows since at $t=\frac{2}{3},\frac{4}{3}$ we have \begin{align*}\Upsilon_{T_{3,3n+1}}(t)-\Upsilon_{T_{2,2n+1}}(t)&=-2n+\frac{2n}{3}
=\frac{-6n+2n}{3}=-\frac{4n}{3},\end{align*}
while at $t=1$, we find \[\U_{T_{3,3n+1}}(t)-\Upsilon_{T_{2,2n+1}}(t)\overset{\text{\eqref{eq:upsilonforindex3torusknots}}}{=}-2n+n=-n.\] Therefore,
\[\Upsilon_{K_{n}}(\frac{2}{3})=\Upsilon_{K_{n}}(\frac{4}{3})\overset{\eqref{eq:upperboundUps(Kn)}}{\leq} -\frac{4n}{3}+\frac{2}{3}
<-n-1\overset{\eqref{eq:lowerboundUps(Kn)}}{\leq} \Upsilon_{K_{n}}(1)\]
for all $n>5$.

We remark that the above argument can be used to $\widetilde{\Upsilon(t)}(\beta_n)$ completely. Indeed,
using~\eqref{eq:lowerboundUps(Kn)} and~\eqref{eq:upperboundUps(Kn)} together with the definition of $\widetilde{\Upsilon(t)}$, we have
\[\widetilde{\Upsilon(t)}(\beta_n)=\widetilde{\Upsilon(t)}((a_1a_2)^{3n})-\widetilde{\Upsilon(t)}(a_1^{2n})=\Upsilon_{T_{3,3n+1}}(t)-\Upsilon_{T_{2,2n+1}}(t).\]
Thus,
\[\widetilde{\Upsilon(t)}(\beta_n)=\left\{\begin{array}{c}
-2tn \for t\leq \frac{2}{3}\\
-2n+tn
\for \frac{2}{3}\leq t\leq 1
\end{array}\right..\]
\end{Example}

Proposition~\ref{prop:sig=ups},
Corollary~\ref{cor:u=sfortsmall}, and Corollary~\ref{cor:u=sfortoruslinks} might bring one to speculate that $\widetilde{\Upsilon(t)}=\frac{\widetilde{\s_{e^{t\pi i}}}}{2}$ holds for more general families of braids or at least for all $3$-strands braids. 
However, this is false in general
as the following example shows:
\begin{Example}\label{Ex:a^3b^7}
Let $\beta$ be the $3$-braid $a_1^3a_2^7$ and set $\omega=e^{\frac{3}{4}\pi i}$.
A calculation shows that
$|\widetilde{\s_\omega}(\beta)|>\frac{3}{4}\l(\beta)$,
which yields
$|\frac{\widetilde{\s_\omega}(\beta)}{2}|>
|\widetilde{\Upsilon(\frac{3}{4})}(\beta)|$ since
$|\widetilde{\Upsilon(\frac{3}{4})}(\beta)|\leq \frac{3}{8}\l(\beta)$
follows from the fact that $|\Upsilon(\frac{3}{4})|$ is a lower bound for $\frac{3}{4}g_4$.
To be explicit: this can, for example, be established by calculating $\s_{\omega}\left(\widehat{\beta^{16}}\right)=-128$ and then using
\[\left|\s_{\omega}\left(\widehat{\beta^{16k}}\right)\right|\geq k\left|\s_{\omega}\left(\widehat{\beta^{16}}\right)\right|-2(k-1)=128k-2k+2=126k+2>120k=\frac{3}{4}\l(\beta^{k16}),\]
where the first inequality is a consequence of $\s_{\omega}$ being a quasi-morphism of defect 2 on the braid group on 3 strands. As a consequence, we have
\[\left|\frac{\widetilde{\s_\omega}(\beta)}{2}\right|\geq\left|\lim_{k\to\infty}\frac{126k+2}{32k}\right|\geq \frac{3}{8}\l(\beta)+\frac{6}{32}\geq\left|\widetilde{\Upsilon(\tfrac{3}{4})}(\beta)\right|+\frac{3}{16}.\]

We thank Lukas Lewark for sharing his observations 
concerning $\s_{\omega}\left(\widehat{\beta^{k}}\right)$. 

\end{Example}

\subsection{Proof of Proposition~\ref{prop:sig=ups}}
We now turn to the proof of Proposition~\ref{prop:sig=ups}. Here is a brief outline of the strategy:

Let $\beta$ be a 3-braid which has a knot $K$ as its closure.
In a first step, we will show that (up to performing a small cobordism) $K$ can be written as a connect sum of a torus knot and the closure of a positive 3-braid in which all generators occur with powers of squares and higher. In a second step, we will see that calculating $\Upsilon(1)$ and $\frac{\sigma}{2}$ for these special braids can be reduced to calculations for torus knots of braid index 3 or less. Since for torus knots of braid index 3 or less $\Upsilon(1)$ and $\frac{\sigma}{2}$ agree (up to some global constant), we will conclude that there is a constant $d$ such that $|\Upsilon_K(1)-\frac{\sigma}{2}(K)|\leq d$ for all $K$ that are closures of 3-braids. This in turn will imply Proposition~\ref{prop:sig=ups} by the definition of the homogenization as given in~\eqref{eq:defHomUpsilon}.

\begin{proof}[Proof of Proposition~\ref{prop:sig=ups}]
Let $\beta$ be any $3$-braid. By the definition of homogenization for 3-braids, we need to consider $\beta^na_1a_2$ for large $n$. In the entire proof, we fix $n$ as a positive integer that is a multiple of $6$.

We are interested in $\upsilon$ and $\s$ of $K=\widehat{\beta^na_1a_2}$. The idea of the proof is to rewrite $K$ (up to performing a cobordism of genus 3) as the connected sum of a torus knot and a positive 3-braid in which all generators appear in powers of squares or higher. For the latter, it turns out that there exists small cobordisms to connected sums of torus knots of braid index 3 or less. This will suffice to conclude that $2\upsilon$ and $\s$ agree on $K$ up to a constant that is independent from $n$ and $\beta$ since the same is true on torus knots of braid index 3 or less.

We replace $\beta^na_1a_2$ by braid $(a_1a_2)^{3k}\alpha$ with the same closure, where $\alpha$ is a positive braid and $k$ is an integer that is maximal among all integers $k'$ with the following property: there exists a positive $3$-braid $\alpha'$ such that $(a_1a_2)^{3k'}\alpha'$ has the same closure as $\beta^na_1a_2$.
\begin{Claim}\label{claim:cobtoconnectedsumoftorusknotand}
There is a cobordism of genus $3$ or less from $K$ to $T\sharp K'$ where $T$ is a $T_{3,3k+1}$ torus knot and $K'$ is a knot given as the closure of a $3$-braid of the form
\begin{equation*}\gamma=a_1^{m_1}a_2^{m_2}a_1^{m_3}\cdots a_{2}^{m_{2l}}\end{equation*}
with $m_i\geq 2$ for all $1\leq i\leq 2l$.
In particular,
$|\upsilon(K)
-\upsilon(T\sharp K')|
\leq 3$.\end{Claim}
We delay the proof of this Claim and first apply it.
We aim to show that there is a constant $d$ (independent of $K$) such that
\begin{equation}\label{eq:upsilonK-sK/2isbounded}\left|\upsilon(K)-\frac{\s(K)}{2}\right|\leq d, \end{equation}
which suffices to prove Proposition~\ref{prop:sig=ups} by the definition of $\widetilde{\u}$ and $\widetilde{\frac{\sigma}{2}}$. Indeed,~\eqref{eq:upsilonK-sK/2isbounded} yields that
\[\left|\lim_{n\to\infty}\frac{\upsilon\left(\widehat{\beta^na_1a_2}\right)}{n}-\lim_{n\to\infty}\frac{\s\left(\widehat{\beta^na_1a_2}\right)}{2n}\right|\leq \lim_{n\to\infty}\frac{d}{n}=0\]
for all 3-braids $\beta$.
Equation~\eqref{eq:upsilonK-sK/2isbounded} is established by calculating
$\upsilon(K)$ and $\sigma(K)$ in terms of $k$, $l$, and $\sum_{i=1}^{2l}m_i$ up to a constant that does not depend on $K$.
We start with calculating $\upsilon(K)$; however, in the course of the calculation it will become apparent that this only uses the fact that $|\upsilon|$ is a concordance invariant that bounds the smooth slice genus from below and the values of $\upsilon$ on torus knots of braid index $2$ and $3$. Since $\upsilon$ and $\frac{\s}{2}$ agree on torus knots of braid index $2$ and differ by at most $2$ on torus knots of index $3$, the same calculation (replacing $\upsilon$ by $\frac{\s}{2}$) will yield the same formula for $\frac{\s}{2}$ up to a constant that does not depend on $K$.

We now estimate $\upsilon(K')$. By deleting $\sum_{i=1}^{2l}(m_i-2)$ generators in $\gamma$ and afterwards adding two generators, we can change $\gamma$ to
the $3$-braid $\beta_l=a_1^2a_2^2\cdots a_1^2a_2^2a_1a_2$ of length $4l+2$. In other words, there is a cobordism of genus $\frac{2+\sum_{i=1}^{2l}(m_i-2)}{2}$ from $K'$ to $K_l=\widehat{\beta_l}$. Combining this with~\eqref{eq:lowerboundUps(Kn)} from Example~\ref{Ex:a1a1a2a2}, we find
\begin{equation}\label{eq:K'leq}\upsilon(K')\geq -\frac{2+\sum_{i=1}^{2l}(m_i-2)}{2}-l-1=-\frac{\sum_{i=1}^{2l}m_i}{2}+l-2.\end{equation}

For an upper bound on $\upsilon(K')$, we use the following claim, which we prove at the end of this section.
\begin{Claim}\label{claim:cobK'toconnectsumtorusknots}
Let $\epsilon,\epsilon_i\in\{0,1\}$ be such that $\epsilon +\sum_{i=1}^{l} m_{2i-1}$ is odd and $m_i+\epsilon_i$ is odd for all $i\in\{2,4,6,\cdots,2l\}$.
There exists a cobordism from $K'$ to the knot
\[K''=T_{2,m_1+m_3+\cdots+m_{2l-1}+\epsilon}\sharp T_{2,m_2+\epsilon_2}\sharp T_{2,m_4+\epsilon_4}\sharp\cdots\sharp T_{2,m_{2l}+\epsilon_{2l}}\]
of genus $\frac{\epsilon+(\sum_{i=1}^{2l}\epsilon_{2i})+l-1}{2}$. In particular,
\begin{align*}\upsilon(K')
&\leq\upsilon(K'')+\frac{\epsilon+(\sum_{i=1}^l\epsilon_{2i})+l-1}{2}
\\&=-g(K'')+\frac{(\epsilon+\sum_{i=1}^l\epsilon_{2i})+l-1}{2}
\\&=-\frac{m_1+m_3+\cdots+m_{2l-1}+\epsilon-1+\sum_{i=1}^l(m_{2i}+\epsilon_{2i}-1)}{2}+\frac{\epsilon+(\sum_{i=1}^l\epsilon_{2i})+l-1}{2}
\\&=-\frac{\sum_{i=1}^{2l}m_i}{2}+l.
\end{align*}
\end{Claim}
Using Claim~\ref{claim:cobtoconnectedsumoftorusknotand}, \eqref{eq:K'leq}, and Claim~\ref{claim:cobK'toconnectsumtorusknots}, we calculate that
\begin{equation}\label{eq:upsilonK=-2k+l+summi}
-2k-\frac{\sum_{i=1}^{2l}m_i}{2}+l-6\leq\upsilon(K)\leq -2k-\frac{\sum_{i=1}^{2l}m_i}{2}+l+3.
\end{equation}
Indeed, we have
\begin{align*}
\upsilon(K)
&\geq -3+\upsilon(T\sharp K')
\\&\geq -3+\upsilon(T_{3,3k+1})+\upsilon(K')
\\&\geq -3-2k-1-\frac{\sum_{i=1}^{2l}m_i}{2}+l-2
\\&=-2k-\frac{\sum_{i=1}^{2l}m_i}{2}+l-6,
\end{align*}
where in the first line we used Claim~\ref{claim:cobtoconnectedsumoftorusknotand} and in the second to last line we used~\eqref{eq:K'leq} and~\eqref{eq:upsilonforindex3torusknots}.
Using Claim~\ref{claim:cobK'toconnectsumtorusknots} instead of~\eqref{eq:K'leq},
a similar calculation establishes the second inequality of~\eqref{eq:upsilonK=-2k+l+summi}:
\begin{align*}
\upsilon(K)
&\leq 3+\upsilon(T\sharp K')
\\&\leq 3+\upsilon(T_{3,3k+1})+\upsilon(K')
\\&\leq 3-2k-\frac{\sum_{i=1}^{2l}m_i}{2}+l
\\&=-2k-\frac{\sum_{i=1}^{2l}m_i}{2}+l+3.
\end{align*}

Finally, we note that all the calculations we did above also work for $\frac{\s}{2}$ instead of $\upsilon$ up to a change of constants. This follows from the fact that on torus knots of braid index 3 or less the values of $\frac{\s}{2}$ differ by at most $2$ from $\upsilon$; compare~\eqref{eq:upsilonforindex3torusknots} and~\cite[Proposition~9.1]{Murasugi_OnClosed3Braids}.
In other words,~\eqref{eq:upsilonK=-2k+l+summi} also holds for $\frac{\s}{2}$ instead of $\upsilon$ up to a change of the constants. This yields that \[\left|\upsilon(K)-\frac{\s(K)}{2}\right|=\left|\upsilon\left(\widehat{\beta^na_1a_2}\right)-\frac{\s\left(\widehat{\beta^na_1a_2}\right)}{2}\right|\] is bounded by a constant $d$ that is independent of $\beta$ and $n$.

It remains to prove Claims~\ref{claim:cobtoconnectedsumoftorusknotand} and~\ref{claim:cobK'toconnectsumtorusknots}.
\begin{proof}[Proof of Claim~\ref{claim:cobtoconnectedsumoftorusknotand}]
Recall that we have $k$ maximal such that $\beta^na_1a_2$ has the same closure as $(a_1a_2)^{3k}\alpha$ for some positive $3$-braid $\alpha$. Note that the closure of $\alpha$ is a knot since the closure of $(a_1a_2)^{3k}\alpha$ is. 
Up to conjugation (which does not change the closure of $(a_1a_2)^{3k}\alpha$), we can choose
$\alpha=a_1^{m_1}a_2^{m_2}a_1^{m_3}\cdots a_{2}^{m_{2l}}$ for some positive integer $l$ and positive integers $m_i$. We choose $l$ minimal. So, for example, $\alpha=a_1a_2a_1a_2$ is not considered since it is isotopic to $a_1a_1a_2a_1$, which is conjugate to $a_1a_1a_1a_2$.

With the exception of the case $\alpha=a_1a_2$, which implies $\beta^na_1a_2$ has the same closure as $(a_1a_2)^{3k+1}$, we show that $m_i\geq2$ for all but one $i\in\{1,\cdots, 2l\}$.
If $l=1$, $\alpha=a_1a_2$ is the above mentioned exception. So it remains to consider the case where $l\geq 2$.
Assume towards a contradiction that there exist $i,j$ such that $m_i=m_j=1$.
We first establish that $m_i=m_{i+1}=1$ or $m_1=m_{2l}=1$ is impossible. Indeed, if this were true, we would have $m_2=m_3=1$ or $m_1=m_2=1$ up to conjugation. We discuss the case $m_2=m_3=1$ since the case $m_1=m_2=1$ is similar (in fact, the latter reduces to the former by exchanging $a_1$ and $a_2$ in $(a_1a_2)^{3k}\alpha$, which does not change the closure of $(a_1a_2)^{3k}\alpha$). Since $m_2=m_3=1$, we have
\[\alpha=a_1^{m_1}a_2a_1a_2^{m_4}\cdots=a_1^{m_1+m_4}a_2a_1\cdots,\] which implies that $l$ was not minimal.
Therefore, we have $|i-j|>1$ and $\{i,j\}\neq \{1,2l\}$. This yields that up to cyclic permutation, $\alpha$ equals
\[\Delta\alpha'\Delta\alpha''=\Delta\Delta\overline{\alpha'}\alpha'',\]
where $\alpha'$ and $\alpha''$ are positive 3-braids, $\overline{\alpha'}$ denotes the braid obtained from $\alpha'$ by switching $a_1$ with $a_2$, and $\Delta$ denotes the half-twist $a_1a_2a_1=a_2a_1a_2$. This contradicts the maximality of $k$ since $\Delta\Delta=(a_1a_2)^3$. 

Deleting at most one generator in $\alpha$, we find a positive $3$-braid that up to conjugation equals \[\widetilde{\alpha}=a_1^{m'_1}a_2^{m'_2}a_1^{m'_3}\cdots a_{2}^{m'_{2l'}},\] where $m'_i\geq 2$ for all $i$. By adding at most one generator, we may assume that the closure of $\widetilde{\alpha}$ is a knot.

In conclusion we have that by adding or deleting at most two generators in $(a_1a_2)^{3k}\alpha$, we may assume that all $m_i\geq 2$. Therefore, there exists a genus 1 cobordism from $\beta^na_1a_2$ to $(a_1a_2)^{3k}\gamma$ where $\gamma=a_1^{m_1}a_2^{m_2}a_1^{m_3}\cdots a_{2}^{m_{2l}}$ with $m_i\geq 2$. Furthermore, there is a genus 2 cobordism from $\widehat{(a_1a_2)^{3k}\gamma}$ to $\widehat{(a_1a_2)^{3k+1}}\sharp\widehat{\gamma}$. Indeed, two saddle moves turn $\widehat{(a_1a_2)^{3k}\gamma}$ into $\widehat{(a_1a_2)^{3k}}\sharp\widehat{\gamma}$ and adding two generators to $\widehat{(a_1a_2)^{3k}}\sharp\widehat{\gamma}$ yields $\widehat{(a_1a_2)^{3k+1}}\sharp\widehat{\gamma}$. Combining the two cobordisms, we have a genus $3$ cobordism between $K=\widehat{\beta^na_1a_2}$ and $\widehat{(a_1a_2)^{3k+1}}\sharp\widehat{\gamma}$.
\end{proof}
\begin{proof}[Proof of Claim~\ref{claim:cobK'toconnectsumtorusknots}]
We first observe that by performing $l-1$ saddle moves 
the knot
$K'$
can be turned into the link
\begin{equation}\label{eq:sumsof2-n-toruslinks}T_{2,m_1+m_3+\cdots+m_{2l-1}}\sharp T_{2,m_2}\sharp T_{2,m_4}\sharp\cdots\sharp T_{2,m_{2l}}.\end{equation}
This is illustrated in Figure~\ref{fig:K'toconnectedsum}; compare also~\cite[Proof of Proposition~5]{Feller_15_ASharpSignatureBoundForPositiveFour-Braids}.
\begin{figure}[ht]
\centering
\def\svgscale{1.7}
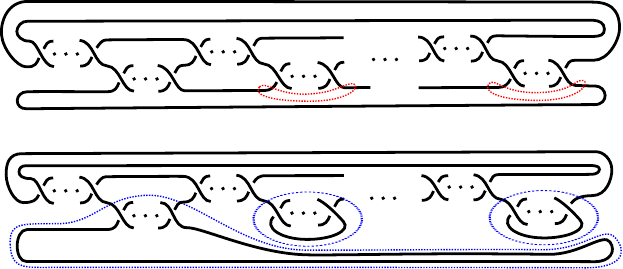
\caption{Top: The knot $K'$ with $l-1$ spheres (red) that indicate where the saddle moves happen.
Bottom: A connected sum of braid index 2 torus links resulting from $l-1$ saddle moves. Splitting spheres (blue) are indicated. 
}
\label{fig:K'toconnectedsum}
\end{figure}
Next we study the summands of~\eqref{eq:sumsof2-n-toruslinks}. Whenever $m_{2i}$ is even, a saddle move turns $T_{2,m_{2i}}$ into $T_{2,m_{2i}+\epsilon_{2i}}$. 
Similarly, if $m_1+m_3+\cdots+m_{2l-1}$ is even, then one saddle move turns $T_{2,m_1+m_3+\cdots+m_{2l-1}}$ into $T_{2,m_1+m_3+\cdots+m_{2l-1}+\epsilon}$. Combined we have that
\[K''=T_{2,m_1+m_3+\cdots+m_{2l-1}+1}\sharp T_{2,m_2+\epsilon_2}\sharp T_{2,m_4+\epsilon_4}\sharp\cdots\sharp T_{2,m_{2l}+\epsilon_{2l}}\]
is obtained from $K'$ by ${\epsilon+(\sum_{i=1}^l\epsilon_{2i})+l-1}$ saddle moves. In particular, there is a cobordism of genus
$\frac{\epsilon+(\sum_{i=1}^l\epsilon_{2i})+l-1}{2}$ between $K'$ and $K''$.
\end{proof}
This concludes the proof of Proposition~\ref{prop:sig=ups}.
\end{proof}

\section{Questions}\label{sec:questions}
\begin{Question}\label{qu:bindex}
Does \[\Ut(K)=-t\tau(K) \for t\leq\frac{2}{n}\] hold for all knots $K$ of braid index $n$ or less?
\end{Question}
Note that, if Question \ref{qu:bindex} can be answered in the positive, then $\Upsilon$ bounds not only the braid index, but the `concordance braid index' of $K$---the minimal braid index of any knot concordant to $K$.
In particular, the quasi-positivity assumption on $K$ in Corollary~\ref{cor:concordancebraidindexdetection} could be dropped.
We now formulate a concordance version of the generalized Jones conjecture, which, if true, also implies that the quasi-positivity assumption on $K$ in Corollary~\ref{cor:concordancebraidindexdetection} could be dropped:
\begin{Question}\label{qu:concordanceGenJonesConj}
For every concordance class $C$ in $\mathfrak{C}$, let $B(C)$ denote the minimal braid index among all the knots in $C$.
Given an $n$-braid $\beta$ and a $B(C)$-braid $\beta'$ both of which have closure in the concordance class $C$, does
\[|\l(\beta)-\l(\beta')|\leq n-B(C)\]
hold for all choices of $C,n,\beta$, and $\beta'$?
\end{Question}

Proposition~\ref{prop:qpbraidindexdetection} (in combination with Lemma~\ref{lem:m_Kforquasipos})
provides a way to detect the braid index
of quasi-positive braid closures.
\begin{Question} Are there examples of quasi-positive knots $L$ where this detects the braid index, not coming from Theorem~\ref{thm:unknottingtotorusknotandbraidindex} (i.e. ~when $L$ does not `contain a full-twist')?
\end{Question}

We invoked the generalized Jones conjecture to be able to use Proposition~\ref{prop:qpbraidindexdetection} to detect braid index. If the following question has a positive answer, then (at least for quasi-positive knots) this can be bypassed.
\begin{Question}\label{qu:braidindex=qpbraidindex}Can the braid index of a quasi-positive knot always be realized as a quasi-positive braid?\end{Question}
It is worth noting that the analog of Question~\ref{qu:braidindex=qpbraidindex} for closures of positive braids was answered in the negative by Stoimenow: there are examples of positive braid knots that have braid index strictly less than the minimal number of strands needed to represent them as positive braids~\cite[Example~7]{Stoimenow_02}.
\begin{Remark}
Question~\ref{qu:braidindex=qpbraidindex} has been answered in the positive by Hayden in~\cite{Hayden16} using braid foliation results from LaFountain and Menasco's proof of the Generalized Jones Conjecture given in~\cite{LaFountainMenasco_14}.
\end{Remark}

\begin{Question}\label{qu:asympt:upsilon=sigma}
Is it true that $\widetilde{\u}=\frac{\widetilde{\s}}{2}$ for all braids?
\end{Question}
For $3$-braids, Proposition~\ref{prop:sig=ups} answers Question~\ref{qu:asympt:upsilon=sigma} in the positive.
By Remark~\ref{rem} and the fact that $\frac{\s}{2}$ is a quasi-morphism with defect $\frac{n-1}{2}$, a positive answer to Question~\ref{qu:asympt:upsilon=sigma} would yield that
\[|\u-\frac{\sigma}{2}|{\leq} n-1\] for the closure of an $n$-braid, which again might be used to detect braid index.
Ozsv\'ath, Stipsicz, and Szab\'o showed that $|\u(K)-\frac{\sigma}{2}(K)|$ is a lower bound for the smooth four-dimensional
crosscap number of a knot $K$\textemdash the minimal first Betti number of (possibly non-oriented) smooth surfaces in $B^4$ with boundary $K\subset S^3$~\cite[Theorem~1.2]{OSS_2015}. If Question~\ref{qu:asympt:upsilon=sigma} has a positive answer, then
 the lower bound for the smooth four-dimensional
crosscap number given by $|\u-\frac{\sigma}{2}|$ can never exceed the braid index minus $1$.
\bibliographystyle{alpha}
\bibliography{upsilonfinal}

\def\cprime{$'$}
\begin{thebibliography}{GLM81}

\bibitem[Ale23]{Alexander_23_ALemmaOnSystemsOfKnottedCurves}
J.~W. Alexander.
\newblock A lemma on systems of knotted curves.
\newblock {\em Proc. Nat. Acad. Sci. USA}, 9:93--95, 1923.

\bibitem[Art25]{Artin_TheorieDerZoepfe}
E.~Artin.
\newblock Theorie der {Z}{\"o}pfe.
\newblock {\em Abh. Math. Sem. Univ. Hamburg}, 4(1):47--72, 1925.

\bibitem[BH15]{BorodzikHedden}
M.~{Borodzik} and M.~{Hedden}.
\newblock The {U}psilon function of {L}-space knots is a {L}egendre transform.
\newblock {\em ArXiv e-prints}, 2015.
\newblock ArXiv:1505.06672 [math.GT].

\bibitem[BL16]{BorodzikLivingston_13Ar}
M.~Borodzik and C.~Livingston.
\newblock Semigroups, {$d$}-invariants and deformations of cuspidal singular
  points of plane curves.
\newblock {\em J. Lond. Math. Soc. (2)}, 93(2):439--463, 2016.
\newblock ArXiv:1305.2868 [math.AG].

\bibitem[BN16]{BodnarNemethi}
J.~Bodn{\'a}r and A.~N{\'e}methi.
\newblock Lattice cohomology and rational cuspidal curves.
\newblock {\em Math. Res. Lett.}, 23(2):339--375, 2016.
\newblock ArXiv:1405.0437 [math.AG].

\bibitem[BO01]{BoileauOrevkov_QuasiPositivite}
M.~Boileau and S.~Orevkov.
\newblock Quasi-positivit\'e d'une courbe analytique dans une boule
  pseudo-convexe.
\newblock {\em C. R. Acad. Sci. Paris S\'er. I Math.}, 332(9):825--830, 2001.

\bibitem[Bra11]{Brandenbursky_11}
M.~Brandenbursky.
\newblock On quasi-morphisms from knot and braid invariants.
\newblock {\em J. Knot Theory Ramifications}, 20(10):1397--1417, 2011.

\bibitem[Dor15]{Dore15}
D.~Dore.
\newblock {R}ecursive properties of the upsilon invariant for torus knots.
\newblock Princeton Undergraduate Thesis, 2015.

\bibitem[DP13]{DynnikovPrasolov_13}
I.~A. Dynnikov and M.~V. Prasolov.
\newblock Bypasses for rectangular diagrams. {A} proof of the {J}ones
  conjecture and related questions.
\newblock {\em Trans. Moscow Math. Soc.}, pages 97--144, 2013.

\bibitem[Fel15a]{Feller_15_MinCobBetweenTorusknots}
P.~Feller.
\newblock Optimal cobordisms between torus knots.
\newblock {\em Communications in Analysis and Geometry}, accepted for
  publication, 2015.
\newblock ArXiv:1501.00483 [math.GT].

\bibitem[Fel15b]{Feller_15_ASharpSignatureBoundForPositiveFour-Braids}
P.~Feller.
\newblock A sharp signature bound for positive four-braids.
\newblock {\em ArXiv e-prints}, 2015.
\newblock ArXiv:1508.00418 [math.GT].

\bibitem[FW87]{Franks_Williams_87_BraidsAndTheJonesPolynomial}
J.~Franks and R.~F. Williams.
\newblock Braids and the {J}ones polynomial.
\newblock {\em Trans. Amer. Math. Soc.}, 303(1):97--108, 1987.

\bibitem[GG05]{GambaudoGhys_BraidsSignatures}
J.-M. Gambaudo and {\'E}.~Ghys.
\newblock Braids and signatures.
\newblock {\em Bull. Soc. Math. France}, 133(4):541--579, 2005.

\bibitem[GLM81]{GLM}
C.~McA. Gordon, R.~A. Litherland, and K.~Murasugi.
\newblock Signatures of covering links.
\newblock {\em Canad. J. Math.}, 33(2):381--394, 1981.

\bibitem[Hay16]{Hayden16}
K.~Hayden.
\newblock Minimal braid representatives of quasipositive links.
\newblock {\em ArXiv e-prints}, 2016.
\newblock ArXiv:1605.01711 [math.GT].

\bibitem[HW14]{HeddenWatson}
M.~{Hedden} and L.~{Watson}.
\newblock On the geography and botany of knot {F}loer homology.
\newblock {\em ArXiv e-prints}, 2014.
\newblock ArXiv:1404.6913 [math.GT].

\bibitem[KM93]{KronheimerMrowka_Gaugetheoryforemb}
P.~B. Kronheimer and T.~S. Mrowka.
\newblock Gauge theory for embedded surfaces. {I}.
\newblock {\em Topology}, 32(4):773--826, 1993.

\bibitem[Krc14]{Krcatovich}
D.~Krcatovich.
\newblock A restriction on the {A}lexander polynomials of {L}-space knots.
\newblock {\em ArXiv e-prints}, 2014.
\newblock ArXiv:1408.3886.

\bibitem[Lev69]{levine}
J.~Levine.
\newblock Knot cobordism groups in codimension two.
\newblock {\em Comment. Math. Helv.}, 44:229--244, 1969.

\bibitem[Liv04]{Livingston_Comp}
Charles Livingston.
\newblock Computations of the {O}zsv{\'a}th-{S}zab{\'o} knot concordance
  invariant.
\newblock {\em Geom. Topol.}, 8:735--742 (electronic), 2004.
\newblock ArXiv:0311036v3 [math.GT].

\bibitem[LM14]{LaFountainMenasco_14}
D.~J. LaFountain and W.~W. Menasco.
\newblock Embedded annuli and {J}ones' conjecture.
\newblock {\em Algebr. Geom. Topol.}, 14(6):3589--3601, 2014.

\bibitem[Mor86]{Morton_86}
H.~R. Morton.
\newblock Seifert circles and knot polynomials.
\newblock {\em Math. Proc. Cambridge Philos. Soc.}, 99(1):107--109, 1986.

\bibitem[Mur74]{Murasugi_OnClosed3Braids}
K.~Murasugi.
\newblock {\em On closed {$3$}-braids}.
\newblock American Mathematical Society, Providence, R.I., 1974.
\newblock Memoirs of the American Mathematical Society, No. 151.

\bibitem[OS03a]{OzsvathSzabo_03_AbsolutlyGradedFloerHomologies}
P.~Ozsv{\'a}th and Z.~Szab{\'o}.
\newblock Absolutely graded {F}loer homologies and intersection forms for
  four-manifolds with boundary.
\newblock {\em Adv. Math.}, 173(2):179--261, 2003.

\bibitem[OS03b]{OzsvathSzabo_03_KFHandthefourballgenus}
P.~Ozsv{\'a}th and Z.~Szab{\'o}.
\newblock Knot {F}loer homology and the four-ball genus.
\newblock {\em Geom. Topol.}, 7:615--639, 2003.

\bibitem[OSS14]{OSS_2014}
P.~Ozsv{\'a}th, A.~I. Stipsicz, and Z.~Szab{\'o}.
\newblock Concordance homomorphisms from knot floer homology.
\newblock {\em ArXiv e-prints}, 2014.
\newblock ArXiv:1407.1795 [math.GT].

\bibitem[OSS15]{OSS_2015}
P.~Ozsv{\'a}th, A.~I. Stipsicz, and Z.~Szab{\'o}.
\newblock Unoriented knot floer homology and the unoriented four-ball genus.
\newblock {\em ArXiv e-prints}, 2015.
\newblock ArXiv:1508.03243 [math.GT].

\bibitem[Rud83]{Rudolph_83_AlgFunctionsAndClosedBraids}
L.~Rudolph.
\newblock Algebraic functions and closed braids.
\newblock {\em Topology}, 22(2):191--202, 1983.

\bibitem[Rud93]{rudolph_QPasObstruction}
L.~Rudolph.
\newblock Quasipositivity as an obstruction to sliceness.
\newblock {\em Bull. Amer. Math. Soc. (N.S.)}, 29(1):51--59, 1993.

\bibitem[Sto02]{Stoimenow_02}
A.~Stoimenow.
\newblock On the crossing number of positive knots and braids and braid index
  criteria of {J}ones and {M}orton-{W}illiams-{F}ranks.
\newblock {\em Trans. Amer. Math. Soc.}, 354(10):3927--3954 (electronic), 2002.

\bibitem[Sto08]{Stoimenow_08}
A.~Stoimenow.
\newblock Bennequin's inequality and the positivity of the signature.
\newblock {\em Trans. Amer. Math. Soc.}, 360(10):5173--5199, 2008.

\bibitem[Tri69]{tristram}
A.~G. Tristram.
\newblock Some cobordism invariants for links.
\newblock {\em Proc. Cambridge Philos. Soc.}, 66:251--264, 1969.

\bibitem[Tro62]{Trotter_62_HomologywithApptoKnotTheory}
H.~F. Trotter.
\newblock Homology of group systems with applications to knot theory.
\newblock {\em Ann. of Math. (2)}, 76:464--498, 1962.

\bibitem[Wal04]{Wall}
C.~T.~C. Wall.
\newblock {\em Singular Points of Plane Curves}, volume~63 of {\em London
  Mathematical Society Students Texts}.
\newblock Cambridge University Press, 2004.

\bibitem[Wan16]{Wang}
S.~Wang.
\newblock On the first singularity for the upsilon invariant of algebraic
  knots.
\newblock {\em Bull. Lond. Math. Soc.}, 48(2):349--354, 2016.
\newblock ArXvi:1505.06835 [math.GT].

\end{thebibliography}

\end{document}